\documentclass[english]{amsart}

\usepackage{amssymb}
\usepackage{epsfig}  		\usepackage{epic,eepic}   
\usepackage{amsmath}
\usepackage{amsxtra,amssymb,latexsym,amscd,amsthm}
\usepackage{bbm}
\usepackage{array}
\usepackage{multirow}
\usepackage{adjustbox}
\usepackage{caption}
\usepackage{mathtools}
\usepackage{ mathrsfs }
\usepackage{fancyhdr}
\usepackage{cases}
\usepackage{textcomp}
\usepackage{tikz}
\usepackage{tikz-cd}
\usepackage{graphicx}
\usepackage{fancybox}
\usepackage{hyperref}
\usepackage[all]{xy}
\usepackage{xcolor}
\usepackage{ tipa }
\usepackage{appendix}
\usepackage[portrait, top=2.5cm, bottom=2.5cm, left=2.5cm, right=2.5cm] {geometry}
\usepackage{bbm}
\usepackage{float}
\restylefloat{table}
\usepackage[backend=biber, style=alphabetic]{biblatex}
\usepackage{stackengine}

\newtheorem{thm}{Theorem}[section]
\newtheorem{prop}[thm]{Proposition}
\newtheorem{lm}[thm]{Lemma}
\newtheorem{cor}[thm]{Corollary}
\newtheorem{conj}[thm]{Conjecture}

\theoremstyle{definition}
\newtheorem{definition}[thm]{Definition}
\newtheorem{ex}[thm]{Example}
\newtheorem{rmk}[thm]{Remark}

\newcommand{\m}{\mathrm{m}}
\newcommand{\reg}{\mathrm{reg}}
\newcommand{\dR}{\mathrm{dR}}
\newcommand{\gm}{\mathrm{gm}}

\newcommand{\rat}{\mathrm{rat}}

\newcommand{\Abb}{\mathbb{A}}
\newcommand{\Cbb}{\mathbb{C}}
\newcommand{\Gbb}{\mathbb{G}}

\newcommand{\Fbb}{\mathbb{F}}
\newcommand{\Pbb}{\mathbb{P}}
\newcommand{\Qbb}{\mathbb{Q}}
\newcommand{\Rbb}{\mathbb{R}}
\newcommand{\Tbb}{\mathbb{T}}
\newcommand{\Zbb}{\mathbb{Z}}
\newcommand{\Lbb}{\mathbb{L}}

\newcommand{\Bscr}{\mathscr{B}}

\newcommand{\Rscr}{\mathscr{R}}

\newcommand{\Bcal}{\mathcal{B}}

\newcommand{\Dcal}{\mathcal{D}}
\newcommand{\Ecal}{\mathcal{E}}

\newcommand{\Hcal}{\mathcal{H}}

\newcommand{\Mcal}{\mathcal{M}}
\newcommand{\Ocal}{\mathcal{O}}

\newcommand{\Scal}{\mathcal{S}}

\newcommand{\Sm}{\operatorname{Sm}}

\newcommand{\DM}{\operatorname{DM}}

\newcommand{\CHM}{\operatorname{CHM}}
\newcommand{\SmProj}{\operatorname{SmProj}}
\newcommand{\Corr}{\operatorname{Corr}}
\newcommand{\id}{\operatorname{id}}

\newcommand{\et}{\textnormal{\'et}}

\newcommand{\SL}{\operatorname{SL}}

\newcommand{\sgn}{\operatorname{sgn}}

\newcommand{\ord}{\operatorname{ord}}

\newcommand{\rk}{\operatorname{rk}}

\newcommand{\Pic}{\operatorname{Pic}}
\newcommand{\NS}{\operatorname{NS}}

\newcommand{\Gal}{\operatorname{Gal}}
\newcommand{\Ker}{\operatorname{Ker}}

\newcommand{\pr}{\operatorname{pr}}

\newcommand{\alg}{\operatorname{alg}}
\newcommand{\tr}{\operatorname{tr}}

\newcommand{\Hom}{\operatorname{Hom}}

\newcommand{\End}{\operatorname{End}}

\newcommand{\Spec}{\operatorname{Spec}}

\renewcommand{\Im}{\operatorname{Im}}

\newcommand{\sHom}{\mathcal{H}\kern -.5pt om}
\newcommand{\sExt}{\mathcal{E}\kern -.5pt xt}

\numberwithin{equation}{subsection}

\begin{document}
	
	\title[The Mahler measure of exact polynomials]{The Mahler measure of exact polynomials\\and special $L$-values of $K3$ surfaces}

	\author{Trieu Thu Ha}
	\address{Unité de Mathématiques Pures et Appliquées, ÉNS de Lyon, 46, allée d’Italie, 69007, France}
    \email{thu-ha.trieu@ens-lyon.fr}
    \address{Faculty of Mathematics and Informatics, Hanoi University of Science and Technology.}
    \email{ha.trieuthu@hust.edu.vn}
	\thanks{This work was performed partially in the framework of the LABEX MILYON (ANR-10-LABX-0070) of Université de Lyon, in the program ``Investissements d'Avenir'' (ANR-11-IDEX-0007) operated by the French National Research Agency (ANR). It was later completed while the author was working at the Faculty of Mathematics and Informatics, Hanoi University of Science and Technology.}

	\begin{abstract}
	We express the Mahler measure of an exact polynomial in arbitrarily many variables in terms of Deligne-Beilinson cohomology.  We then focus on the relationship between the Mahler measure of four-variable exact polynomials and the special value of the $L$-function of $K3$ surfaces at $s = 4$. This result extends the three-variable case studied in \cite{Tri23}. Finally, we prove, under Beilinson's conjecture, that the Mahler measure of the polynomial $(x+1)(y+1)(z+1) + t$ is expressed in terms of the Riemann zeta function and the $L$-function of the modular form of weight 3 and level 7.
	\end{abstract}

	\maketitle

	\setcounter{section}{-1}
	\section{Introduction}
	
	The Mahler measure of polynomials was introduced by Mahler in 1962 to study transcendental number theory (see  \cite{Mah62}). The (logarithmic) Mahler measure of a polynomial $P(x_1, \dots, x_n) \in \Cbb[x_1^{\pm 1}, \dots, x_n^{\pm 1}]$ is defined by 
	\begin{equation}\label{19}
		\begin{aligned}
			\m(P) &= \dfrac{1}{(2\pi i)^n} \int_{\Tbb^n} \log |P(x_1, \dots, x_n)| \ \dfrac{d x_1}{x_1} \wedge \cdots \wedge \dfrac{d x_n}{x_n},
		\end{aligned}
	\end{equation}
	where $\Tbb^n : |x_1|=\dots = |x_n|=1$ is the $n$-dimensional torus.
	
	In 1997, Deninger \cite{Den97} linked the Mahler measure of a complex polynomial $P$ under certain conditions to the Deligne-Beilinson cohomology of $V_P$, where $V_P$ is the zero locus of $P$ in $(\Cbb^\times)^n$. Deninger defined a chain $\Gamma$ attached to $P$ and an $(n-1)$-differential form $\eta(x_1, \dots, x_n)$  which represents the regulator  $\reg^{n, n}_{\Gbb_m^n}(\{x_1, \dots, x_n\})$, here 
    $$\reg^{n, n}_{\Gbb_m^n} : H^n_\Mcal(\Gbb_m^n, \Qbb(n)) \to H^n_\Dcal(\Gbb_m^n, \Rbb(n))$$
    is the Beilinson regulator map from the motivic cohomology to the Deligne-Beilinson cohomology of $\Gbb_m^n$, and $\{x_1, \dots, x_n\}$ denotes the Milnor symbol in $H^n_\Mcal(\Gbb_m^n, \Qbb(n))$. Deninger showed that if $\Gamma$ is contained in the regular locus $V_P^\reg$ of $V_P$, then
	\begin{equation}\label{1}
		\m(P) = \m(\tilde{P}) + \dfrac{(-1)^{n-1}}{(2\pi i)^{n-1}} \int_\Gamma \eta(x_1,\dots, x_n),
	\end{equation}
	where $\tilde{P}$ is the leading coefficient of $P$ seen as a polynomial in $x_n$.

 Now we assume that $P$ has rational coefficients. The variety $V_P$ is therefore defined over $\Qbb$ and we still denote it by $V_P$.  Deninger found that if $\partial \Gamma  = \varnothing$, the identity \eqref{1} together with Beilinson's conjectures should imply that $\m(P)$ is expressed in terms of the $L$-function of the pure motive $h^{n-1}(\overline{V_P})$, where $\overline{V_P}$ is a smooth compactification of $V_P$. When $\partial \Gamma \neq \varnothing$, Maillot \cite{Mai03} suggested looking at the variety $W_P := V_P \cap V_{P^*}$, where $P^* (x_1,\dots ,x_n) = P(x_1^{-1}, \dots, x_n^{-1})$. As $P$ has rational coefficients, $W_P$ is also defined over $\Qbb$, and we call it the \textit{Maillot variety}. The polynomial $P$ is called \textit{exact} if the regulator $\reg_{V_P^\reg}^{n,n}(\{x_1, \dots, x_n\})$ is trivial, which is equivalent to saying that  $\eta|_{V_{P}^{\reg}}$ is an exact form. Now suppose that $P$ is an exact polynomial. Then Stokes' theorem gives
	$$\m(P) = \m(\tilde{P}) + \dfrac{(-1)^{n-1}}{(2\pi i)^{n-1}} \int_{\partial \Gamma} \omega.$$
    Since $P$ has rational coefficients, $\partial \Gamma$ is contained in $W_P$. One then expects that $\m(P)$ is related to the cohomology of $W_P$ under some conditions. Continuing this direction, Lalín proved the following numerical identity of Boyd \cite{Boy06} under  Beilinson's conjectures
	\begin{equation*}
		\m((x+1)(y+1)+z) = a \cdot L'(E_{15},-1) \quad (a \in \Qbb^\times),
	\end{equation*}
	where $E_{15}$ is the elliptic curve of conductor 15. This identity was then completely proved by Brunault \cite{Bru23}. Lalín's result was generalized by the author in \cite{Tri23}. More precisely, we showed that the Mahler measure of three-variable exact polynomials can be related to the Beilinson regulator of certain motivic cohomology classes. Then under Beilinson's conjecture for genus 1 curves, the Mahler measure can be expressed in terms of the special value of  elliptic curve $L$-functions at $s = 3$ and a linear combination of Bloch-Wigner dilogarithmic values (see \cite[Theorem 0.2]{Tri23}). In \cite{Tri23}, we applied this method and obtained many numerical identities of Boyd and Brunault (see \cite{Boy06}, \cite{Bru20}) under Beilinson's conjectures, for example,
	\begin{equation}
		m(1+ (x^2 - x + 1)y + (x^2 + x + 1)z) = a \cdot L'(E_{45}, -1)  + b \cdot  L'(\chi_{-3},-1) \quad (a \in \mathbb{Q}^\times, b \in \mathbb{Q}^\times), 
	\end{equation}
	\begin{equation*}
		\m(x^2+1 +(x+1)^2 y + (x^2 -1)z) = c \cdot L'(E_{48}, -1) +  L'(\chi_{-4}, -1) \quad (c \in \Qbb^\times).
	\end{equation*}

	In this article, we express the Mahler measure of an exact polynomial in arbitrarily many variables in terms of the Deligne-Beilinson cohomology of an open subset of the Maillot variety (see Theorem \ref{7.2.8}). More precisely, let $P(x_1, \dots, x_n) \in \Qbb[x_1, \dots, x_n]$ be an exact polynomial. If $\partial \Gamma$ does not cross the singularities of $W_P$, it then defines an element in the singular homology group $H_{n-2}(Y(\Cbb), \Qbb)^{(-1)^{n-1}}$ of a smooth open subvariety $Y$ of $W_P$, and $+/-$ respectively denotes the invariant/anti-invariant part under the action of the complex conjugation. We construct a $(n-2)$-differential form $\rho$ which defines an element in Deligne-Beilinson cohomology $H_\Dcal^{n-1}(Y_\Rbb, \Rbb(n))$ such that 
	\begin{equation}\label{296}
		\m (P) = \m(\tilde P) +   \dfrac{(-1)^{n-1}}{(2\pi i)^{n-1}} \left<[\partial \Gamma], [\rho] \right>_Y,
	\end{equation}
	where the pairing is given by 
	\begin{equation}\label{pair1}
		\begin{aligned}
			\left<\ , \ \right>_Y : H_{n-2}(Y(\Cbb), \Qbb)^{(-1)^{n-1}} \times H^{n-2}(Y(\Cbb), \Rbb(n-1))^+ &\to \Rbb(n-1), \quad 
			(\gamma, \omega) &\mapsto \int_\gamma \omega.
		\end{aligned}
	\end{equation}

Let $X$ be a smooth compactification of $W_P$ and $F$ be its function field. We construct an element in the cohomology group $H^{n-1}(\Gamma(F, n))$, where $\Gamma(F, n)$ is Goncharov's polylogarithmic complex (see Lemma \ref{Hn-1Fn}). In the case $n=4$, under an explicit condition, this element gives rise to an element in the cohomology group $H^3(\Gamma(X, 4))$, where $\Gamma(X, 4)$ is Goncharov's polylogarithmic complex of X (see Remark \ref{conditionLambda}).  Under Goncharov's conjecture
    \begin{equation}\label{Gonconj}
		H^3(\Gamma(X, 4)) \stackrel{?}{\simeq} H^3_\Mcal(X, \Qbb(4)),
	\end{equation}
    it then gives us an element, denoted by $\Lambda$, in the  motivic cohomology group $H^3_\Mcal(X, \Qbb(4))$. Then the Mahler measure is related to the Beilinson regulator of $\Lambda$ (see Theorem \ref{4varsmain})
	\begin{equation}\label{303}
		\m(P) = \m(\tilde P) -\dfrac{1}{(2\pi i)^3} \left<[\partial \Gamma], \reg^{3, 4}_X  (\Lambda)\right>_X,
	\end{equation}
	where $\reg^{3, 4}_X : H^3_\Mcal(X, \Qbb(4)) \to H^3_\Dcal(X_\Rbb, \Rbb(4))$ is the Beilinson regulator map.

	Recall that, in the case of three-variable polynomials, if the Maillot variety is a curve of genus one, the Mahler measure can be related to the special value at $s=3$ of the $L$-function of the associated elliptic curve. Therefore, in the four-variable context, it is natural to consider cases in which the Maillot variety admits a projective regular model $X$ that is a $K3$ surface. In particular,  if the projective regular model $X$ is a singular $K3$ surface with Picard rank 20 over $\Qbb$ (i.e., $NS(X_{\bar \Qbb})$ is generated by 20 divisors defined over $\Qbb$), then under Beilinson's conjecture for the transcendental part of motivic cohomology \ref{242}, the Mahler measure is expressed in terms of the Riemann zeta function and the $L$-function of a weight 3 modular form (see Corollary \ref{final})
\begin{equation}\label{eqfinal}
\m(P) = \m(\tilde{P}) + a\cdot L'(f, -1) + b \cdot \zeta'(-2), \quad a, b \in \Qbb.
\end{equation}
Applying this method to the four-variable exact polynomial $P = (x+1)(y+1)(z+1)+t$, we obtain the following result. This appears to be the first known identity relating the Mahler measure of a four-variable polynomial to the $L$-function of a modular form.

	\begin{thm}\label{305} Let $P = (x+1)(y+1)(z+1)+t$. The minimal regular proper model of $W_P$ is a singular $K3$ surface with Picard rank 20 over $\Qbb$ . Assuming Goncharov's Conjecture \ref{Gonconj} and Beilinson's conjecture for the transcendental part of the motivic cohomology of singular $K3$ surfaces, we have
		$$\m(P) = a \cdot L'(f_7, -1) + b \cdot  \zeta'(-2), \quad a, b \in \Qbb,$$
		where $f_7$ is the modular form of weight 3 and level 7.
	\end{thm}
	
	The article is divided into three sections. In Section \ref{MahlerDeligne}, we  prove equation \eqref{296}, which expresses the Mahler measure of exact polynomials in terms of Deligne-Beilinson cohomology.  In Section \ref{section 1.3}, we consider the four-variable case, and prove equation \eqref{303} under Goncharov's conjecture \ref{Gonconj}. In Section \ref{section K3}, we first study the motivic cohomology and Deligne-Beilinson cohomology of $K3$ surfaces and then prove \eqref{eqfinal}. In Section \ref{ex4var}, we prove Theorem \ref{305}.
	\vspace{0.2cm}
	
	\textbf{Acknowledgement.} I would like to thank my PhD supervisor, François Brunault, for many fruitful discussions. I would like to thank Odile Lecacheux and Marie José Bertin for the discussions on fibrations of elliptic surfaces.

	\section{Mahler measure and Deligne-Beilinson cohomology}\label{MahlerDeligne}
	
	In this section, we discuss the relationship between the Mahler measure of exact polynomials and Deligne-Beilinson cohomology. In fact, it generalizes the three-variable case studied in \cite[Lemma 4.8]{Tri23}.  We first recall briefly the definition of Deligne-Beilinson cohomology, motivic cohomology, and Goncharov's  polylogarithmic complexes.  The main result of this section is the explicit construction of an element in the $(n-1)$-cohomology group of Goncharov’s weight-$n$ complex, and the relation of its regulator with the Mahler measure (see Lemma \ref{Hn-1Fn} and Lemma \ref{7.2.7}). 

	\subsection{Deligne-Beilinson cohomology} 
We recall the definition of real Deligne-Beilinson cohomology of Burgos \cite{Bur97}. Let $X$ be a smooth complex variety. Let $j \ge 1$ be an integer.  Let $(\bar X, \iota)$ be a good compactification of $X$, i.e., $\iota : X \to \bar X$ is an embedding where $\bar X$ is a smooth proper variety and $D := \bar X \setminus \iota(X)$  is a normal crossing divisor on $\bar X$. For $\Lambda \in \{\Rbb, \Cbb\}$, we denote by $E^n_{X, \Lambda}(\log D)$ the space of $\Lambda$-valued smooth differential forms on $X$ of degree $n$ with logarithmic singularities along $D$, namely, differential forms generated by smooth forms on $\bar X$, $\log |z_i|, \frac{d z_i}{z_i}$,  $\frac{d \bar z_i }{\bar z_i}$, for $1 \le i \le m$, where $z_1 \dots z_m = 0$ is the local equation of $D$. When $\Lambda = \Cbb$, we have the following decomposition 
	\begin{equation}\label{51}
		E^n_{X, \Cbb} (\log D) = \bigoplus_{p+q = n} E^{p, q}_{X, \Cbb}(\log D),
	\end{equation}
	where $E^{p, q}$ is the subspace of $(p, q)$-forms. Denote by $ E^\bullet_{\log, \Lambda}(X) = \varinjlim_{(\bar X, D)} E^\bullet_{X, \Lambda}(\log D)$. For any integers $j, n \ge 0$, Burgos defined the following complex
	\begin{equation}\label{54}
		E_j(X)^n := \begin{cases}
			(2 \pi i)^{j-1} E^{n-1}_{\log, \Rbb}(X) \cap \left(\bigoplus_{p+q = n-1; p, q < j} E^{p, q}_{\log, \Cbb}(X)\right) & \text{if } n \le 2j-1,\\
			(2\pi i)^j E^n_{\log, \Rbb} (X) \cap \left(\bigoplus_{p+q = n; p, q \ge j} E^{p, q}_{\log, \Cbb}(X)\right) & \text{if } n \ge 2j,
		\end{cases}
	\end{equation}
	with differentials
	\begin{equation*}
		d^n \omega :=\begin{cases}
			-\text{pr}_j(d \omega) & \text{if } n < 2j-1,\\
			-2 \partial\bar \partial \omega & \text{if } n = 2j -1,\\
			d \omega & \text{if } n \ge 2j,
		\end{cases}
	\end{equation*}
	where $\text{pr}_j$ is the projection $\bigoplus_{p, q} \to \bigoplus_{p, q< j}$. 
	Burgos (\cite[Corollary 2.7]{Bur97}) then showed that 
	\begin{equation}\label{52}
		H^n_\Dcal (X, \Rbb(j)) \simeq H^n (E_j(X)^\bullet)\quad \text{for } j, n\ge 0.
	\end{equation}
	If $j > \dim X \ge 1$ or $j > n$, we have the following canonical isomorphism
	\begin{equation}\label{2.1.8}
		H^n_\Dcal(X, \Rbb(j)) \simeq H^{n-1}(X, \Rbb(j-1)),
	\end{equation}
	(see e.g. \cite[Lemma 2.1.10]{Tri24}). In particular, for $j > 1$, we have $H^1_\Dcal(\Spec (\Cbb), \Rbb(j)) \simeq \Rbb(j-1)$.

	Let $X$ be a regular algebraic variety over $\Rbb$. The set of complex points $X(\Cbb)$ is endowed with an action induced by the complex conjugation $\Cbb \to \Cbb, z \mapsto \bar z$, denoted by $F_\infty$. It induces an involution, denoted by $F_\dR$, on differential forms of $X(\Cbb)$ as below
	\begin{equation}\label{82}
		F_{\text{dR}}(\omega) := F^*_\infty (\bar \omega).
	\end{equation}
    The Deligne-Beilinson cohomology of $X$ is defined by the invariant part of the Deligne-Beilinson cohomology of $X(\Cbb)$ under the involution $F_\dR$
	\begin{equation}
		H^n_\Dcal (X/_\Rbb, \Rbb(j)) := H^n_\Dcal(X(\Cbb), \Rbb(j))^{+} \quad \text{for } n, j\ge 0.
	\end{equation}

    \begin{ex}\label{Deligne coho of number field}
		Let $k$ be a number field and $n> 1$ be an integer. Let $r_1$ and $2r_2$ be the numbers of real and complex embeddings from $k$ to $\Cbb$, respectively. As $k \otimes_\Qbb \Cbb \simeq \oplus_{\sigma : k \hookrightarrow \Cbb} \Cbb$, we have
		\begin{equation*}
			\begin{aligned}
				H^1_\Dcal(\Spec (k \otimes_\Qbb \Cbb)/\Rbb, \Rbb(n)) &= 
				\bigoplus_{\sigma : 
					k \hookrightarrow \Cbb} H^1_\Dcal(\Spec(k \otimes_\sigma\Cbb), \Rbb(n))^+\\
				&= \bigoplus_{\sigma : 
					k \hookrightarrow \Cbb} H^0(\Spec(k \otimes_\sigma \Cbb), \Rbb(n-1))^+ \\
				&= \oplus_{\sigma: k \hookrightarrow \Cbb} \begin{cases}  H^0(\Spec(k \otimes_\sigma \Cbb), \Rbb)^- & \text{ if } n \text{ even}\\
					H^0(\Spec(k \otimes_\sigma \Cbb), \Rbb)^+ & \text{ if } n \text{ odd}
				\end{cases}  \\
				&\simeq 
				\begin{cases}
					\Rbb^{r_2} & \text{ if } n \text{ even},\\
					\Rbb^{r_1 + r_2} & \text{ if } n \text{ odd}.
				\end{cases}    
			\end{aligned}
		\end{equation*}
	\end{ex}

\subsection{Motivic cohomology} In this section, let us recall briefly the definition of motivic cohomology. Let $k$ be a field of characteristic 0. Recall that Grothendieck constructed  the category $\CHM(k, \Qbb)$ of Chow motives with $\Qbb$-coefficient and a contravariant functor $h : \SmProj(k) \to \CHM(k, \Qbb)$ which sends a smooth projective variety $X$ to the Chow motive $h(X)$. For example, let us consider the Chow motive of the projective line. Denote by $\pi : \Pbb^1_k \to \Spec k$ the structure morphism and $x : \Spec k \to \Pbb^1_k$ any rational point. Let $\pi^* : h(\Spec k) \to h(\Pbb^1_k)$ and $x^* : h(\Pbb^1_k) \to h(\Spec k)$ be the corresponding morphisms in $\CHM(k, \Qbb)$. Since $x^* \circ \pi^* = \id$, the map $\pi^*$ defines a subobject $\mathbbm{1} := h(\Spec k)$ of $h(\Pbb^1_k)$. As the category $\CHM(k, \Qbb)$ is pseudo-abelian, we get the following canonical decomposition
\begin{equation}\label{P1}
    h(\Pbb^1_k) = \mathbbm{1} \oplus \Lbb,
\end{equation}
where $\Lbb$ is called the Lefschetz motive. 

Voevodsky constructed the category $\DM_\gm(k, \Qbb)$ of \textit{geometrical motives} over $k$ with $\Qbb$-coefficient, and a covariant functor 
\begin{equation}
    M : \Sm(k) \to \DM_\gm(k, \Qbb),
\end{equation}
which sends a smooth variety $X$ to the \textit{motive} $M(X)$ (see \cite[Lecture 14]{MVW06}).  There is a fully faithful contravariant functor from $\CHM(k, \Qbb)$ to  $\DM_\gm(k, \Qbb)$, which sends the Chow motive $h(X)$ of a smooth projective variety $X$ to the motive $M(X)$ (see \cite[Lecture 20]{MVW06}). The corresponding decomposition of $M(\Pbb^1_k)$ in $\DM_\gm(k, \Qbb)$ is given by
$$M(\Pbb^1_k) = \Qbb(0) \oplus \Qbb(1)[2],$$
where $\Qbb(0) := M(\Spec k)$ and $\Qbb(1)$ is called the Tate motive. Denote by $\Qbb(j) = \Qbb(1)^{\otimes j}$.
\begin{definition}[{\textbf{Motivic cohomology}, \cite[Definition 14.17]{MVW06}}]
    Let $X$ be a smooth variety over $k$. Let $n, j \in \Zbb$. The motivic cohomology of $X$ with $\Qbb$-coefficient is defined by 
    \begin{equation*}
        H^n_\Mcal(X, \Qbb(j)) = \Hom_{\DM_\gm(k, \Qbb)} (M(X), \Qbb(j)[n]).
    \end{equation*}
\end{definition}

Let $X$ be a smooth variety over $\Rbb$ or $\Cbb$. Recall that Beilinson defined a $\Qbb$-linear map (see e.g., \cite{Nek13}) 
\begin{equation}\label{reg}
    \reg^{n, j}_X : H^n_\Mcal (X, \Qbb(j)) \to H^n_\Dcal(X, \Rbb(j)).
\end{equation}
When $X$ is defined over a number field $k$,  the Beilinson regulator map is the composition 
\begin{equation}\label{282}
     H^n_\Mcal(X, \Qbb(j)) \xrightarrow{\text{base change}} H^n_\Mcal(X_\Rbb, \Qbb(j)) \xrightarrow{\reg_{X_\Rbb}} H^n_\Dcal (X_\Rbb, \Rbb(j)),
\end{equation}
where $X_\Rbb := X \times_\Qbb \Rbb$. In Section \ref{section 1.3}, we will construct an element in motivic cohomology and relate its Beilinson regulator to the Deligne-Beilinson cohomology element constructed in Section \ref{main1} below.

	\subsection{Goncharov polylogarithmic complexes}\label{Gon}

	Let $F$ be a field of characteristic 0. Goncharov \cite{Gon95} defined the groups $\Bscr_n(F)$ for $n \ge 1$, to be the quotient of the $\Qbb$-vector space $\Qbb[\Pbb_F^1]$ by a certain (inductively defined) subspace $\Rscr_n(F)$
	\begin{equation}
		\Bscr_n(F) := \Qbb[\Pbb_F^1] / \Rscr_n(F) \text{ for } n \ge 1,
	\end{equation}
	By definition, $\Rscr_1(F) := \left<\{x\} + \{y\} - \{xy\}, x, y \in F^\times; \{0\}; \{\infty\}\right>$. De Jeu showed that $\Rscr_2(F)$ is generated by explicit relations (see \cite[Remark 5.3]{dJ00})
	\begin{equation}\label{130}
		\begin{aligned}
			\Rscr_2(F) = \left<\{x\} + \{y\} + \{1-xy\} + \left\{\frac{1-x}{1-xy}\right\} + \left\{\frac{1-y}{1-xy}\right\}, x, y \in F^\times\setminus \{1\}; \{0\}; \{\infty\}\right>.
		\end{aligned}
	\end{equation}
	Goncharov constructed a polylogarithmic complex, denoted by $\Gamma_F(n)$ (see \cite[Section 1.9]{Gon95})
	\begin{equation}\label{weightnGoncomplex}
		\Bscr_n(F) \xrightarrow{\alpha_n(1)} \Bscr_{n-1}(F) \otimes F^\times_\Qbb \xrightarrow{\alpha_n(2)} \Bscr_{n-2}(F) \otimes \bigwedge^2 F^\times_\Qbb \to \cdots \to \Bscr_2(F) \otimes \bigwedge^{n-2} F^\times_\Qbb \xrightarrow{\alpha_n(n)} \bigwedge^n F^\times_\Qbb,
	\end{equation}
	where differentials are given by
	\begin{equation*}
		\{x\}_p \otimes y_1 \wedge \dots \wedge y_{n-p} \mapsto \{x\}_{p-1} \otimes x \wedge y_1 \wedge \dots \wedge y_{n-p} \qquad \text{if } p > 2,
	\end{equation*}
	and 
	\begin{equation*}
		\{x\}_2 \otimes y_1 \wedge \dots \wedge y_{n-2} \mapsto (1-x)\wedge x \wedge y_1 \wedge \dots \wedge y_{n-2}.
	\end{equation*}
	We have $H^n (\Gamma(F, n)) \simeq H^n_\Mcal(F, \Qbb(n))$ by Matsumoto's theorem and  $H^1(\Gamma_F(2)) \xrightarrow{\simeq} H^1_\Mcal(F, \Qbb(2))$ by Suslin's theorem  (see \cite{Sus91}). Goncharov conjectured that (see \cite[ Conjecture A, p.222]{Gon95})
	\begin{equation*}
		H^p(\Gamma(F, n)) \simeq H^p_\Mcal(F, \Qbb(n)), \quad p, n \ge 1.
	\end{equation*}

	Let $k$ be a field of characteristic 0. Let $X$ be a smooth algebraic variety over $k$. Denote by $F = k(X)$ the function field of $X$. Let  $\Ecal ^j (U)$ be the space of real smooth $j$-forms on $U(\Cbb)$, the set of complex points of $U \times_\Qbb \Cbb$. Denote by $\Ecal^j(\eta_X) = \varinjlim_{U \subset X \text{ open}} \Ecal^j(U)$, where $\eta_X$ is the generic point of $X$, and $d$ is the de Rham differential. We denote by $\Ecal^j(\eta_X)(n-1) := \Ecal^j(\eta_X) \otimes \Rbb(n-1)$.
	There is a homomorphism of complexes (see {\cite[Theorem 2.2]{Gon98}})
	\begin{equation}\label{10.5.1}
		\xymatrix{\Bscr_n(F) \ar[rr]^{\alpha_n(1)\hspace{0.8cm}} \ar[d]_{r_n(1)} && \Bscr_{n-1} (F) \otimes F^\times_\Qbb \ar[rr]^{\hspace{1cm}\alpha_n(2)} \ar[d]^{r_n(2)} && \cdots \ar[r]^{\alpha_n(n)} & \bigwedge^n F^\times_\Qbb \ar[d]^{r_n(n)}\\
			\Ecal^0(\eta_X)(n-1) \ar[rr]^{d} && \Ecal^1(\eta_X)(n-1) \ar[rr]^{\qquad d} && \cdots \ar[r]^{d\hspace{1.5cm}} & \Ecal^{n-1}(\eta_X)(n-1).}
	\end{equation}
	In particular, Goncharov gave explicit formulas
	\begin{equation}\label{gonform}
		\begin{aligned}
			&r_{n}(n-1)(\{f\}_2 \otimes g_1 \wedge \cdots \wedge g_{n-2}) := \\
			&\hspace{4cm} i D(f)  \mathrm{Alt}_{n-2} \left( \sum_{p=0}^{[\frac{n-2}{2}]} c_{p, n-1} \bigwedge_{k=1}^{2p} d \log |g_k| \wedge \bigwedge_{\ell=2p+1}^{n-2} d i \arg g_\ell \right)\\
			&\hspace{4cm}+ \theta(1-f, f) \mathrm{Alt}_{n-2} \left( \sum_{m=1}^{[\frac{n-1}{2}]} \dfrac{c_{m-1, n-2}}{(2m+1)} \log |g_1| \bigwedge_{k=2}^{2m-1}d \log|g_k| \wedge\bigwedge_{\ell = 2m}^{n-2} di \arg g_\ell\right),\\
			&r_n(n) (g_1\wedge\cdots \wedge g_n) :=\\
			&\hspace{2cm}\mathrm{Alt}_n \left(\sum_{j=0}^{\left[\frac{n-1}{2}\right]} c_{j,n} \log|g_1| d \log |g_2|\wedge \cdots \wedge d \log|g_{2j+1}| \wedge d i \arg(g_{2j+2}) \wedge \cdots \wedge d i \arg(g_n) \right),
		\end{aligned}
	\end{equation}
	where $D : \Pbb^1(\Cbb) \to \Rbb$ is the Bloch-Wigner dilogarithm function
	\begin{equation}\label{dilog}
		D(z) = \begin{cases}
			\mathrm{Im}\left(\sum_{k=1}^\infty \frac{z^k}{k^2}\right) + \arg (1-z) \log |z| &(|z| \le 1),\\
			-D(1/z) &(|z| \ge 1),
		\end{cases}
	\end{equation}
	and
	\begin{equation}\label{theta}
		\begin{aligned}
			\theta(f, g) &:=  \log|f| d \log |g| - \log |g| d \log |f|, \\
			c_{j,n} &:= \dfrac{1}{(2j+1)!(n-2j-1)!},\\
			\mathrm{Alt}_n G(x_1, \dots, x_n) &:= - \sum_{\sigma \in S_n} \sgn(\sigma) G(x_{\sigma(1)}, \dots, x_{\sigma(n)}).
		\end{aligned}
	\end{equation}

	\begin{cor}[Goncharov]\label{gonreg}
		Let $X$ be a smooth algebraic variety over  a field $k$ of characteristic 0. Denote by $F = k(X)$ the function field of $X$. The maps $r_n(i)$ for $i \le n$ induce maps on cohomology groups
		\begin{equation}\label{rniF}
			r_n(i)_F : H^i (\Gamma(F, n)) \to H^i_\Dcal(F, \Rbb(n)),
		\end{equation}
		where $H^i_\Dcal(F, \Rbb(n)) := \varinjlim_{U \subset X \text{ open }} H^i_\Dcal (U_\Rbb, \Rbb(n))$ and $H_\Dcal^i (U_\Rbb, \Rbb(n))$ denotes the Deligne-Beilinson cohomology of $U_\Rbb := U \otimes_\Qbb \Rbb.$ 
	\end{cor}
	\begin{proof}
		The following complex in degree 1 to $n-1$ 
		$$ \Ecal^0(\eta_X)(n-1) \xrightarrow{d} \Ecal^1(\eta_X)(n-1) \xrightarrow{d} \cdots \xrightarrow{d} \Ecal^{n-1}(\eta_X)(n-1)$$
		computes the weight $n$ Deligne cohomology of the function field from degree 1 to $n-1$ (see \cite{Bur97}, \cite{BZ20}, \cite[Chapter 2]{Tri24}). Therefore, the homomorphism of complexes  $r_n(i)$ induces maps on cohomology groups
		$$r_n(i)_F : H^i (\Gamma(F, n)) \to H^i_\Dcal(F, \Rbb(n)) \quad \text{ for } i = 1, \dots, n-1.$$ In degree $n$, we have that $r_n(n)(f_1\wedge\cdots\wedge f_n)$ defines an element in $H^n_\Dcal(F, \Rbb(n))$ since 
		\begin{equation*}
			(\pr_{n} \circ d) (r_n(n)(f_1\wedge\cdots\wedge f_n)) = \pr_n (\pi_n(d \log f_1 \wedge \cdots \wedge d \log f_n)) = 0,
		\end{equation*}
		where $\pr_n : \oplus_{p, q} \to \oplus_{p, q < n}$.
	\end{proof}

	\subsection{Mahler measure and Deligne-Beilinson  cohomology}\label{main1}
	Let $P \in \Cbb[x_1, \dots, x_n]$ be an irreducible polynomial. Let $V_P = \{(x_1, \dots, x_n) \in (\Cbb^\times)^k : P(x, \dots, x_n) = 0\}$ be the zero locus of $P$ in $(\Cbb^\times)^k$ and let $V_P^\text{reg}$ be its smooth part. Deninger defined
	\begin{equation}\label{Denchain}
		\Gamma = V_P \cap \{(x_1, \dots, x_n) \in (\Cbb^\times)^n : |x_1| = \dots = |x_{n-1}|= 1, |x_n|\ge 1\},
	\end{equation}
	which is a real manifold of dimension $n-1$ with the orientation induced from the torus. By Jensen's formula, we have
	\begin{equation}\label{7}
		\m(P) = \m(\tilde P) + \dfrac{1}{(2\pi i )^{n-1}}\int_{\Gamma} \log|x_n| \ \dfrac{d x_1}{x_1} \wedge \cdots \wedge \dfrac{d x_{n-1}}{x_{n-1}},
	\end{equation}
	where $\tilde{P}\in \Cbb[x_1, \dots, x_{n-1}]$ is the leading coefficient of $P$ by considering $P$ as a polynomial in $x_n$. Deninger introduced the following differential form defined on $\Gbb_m^n$  
	\begin{equation}\label{Denform}
		\eta(x_1, \dots, x_n) = \dfrac{2^{n-1}}{n!} \sum_{\sigma \in \mathcal{S}_n} \mathrm{sgn}(\sigma) \sum_{j=1}^n (-1)^{j-1} \varepsilon_{\sigma_1} \  \bar 
		\partial \varepsilon_{\sigma_2} \wedge \cdots  
		\wedge \bar \partial \varepsilon_{\sigma_j} \wedge \partial \varepsilon_{\sigma_{j+1}} \wedge \cdots \wedge
		\partial \varepsilon_{\sigma_n},
	\end{equation}
	where $\varepsilon_j = \log|x_j|$. Deniner then showed that $\eta (x_1, \dots, x_n)$ is a representative of the regulator $\reg^{n, n}_{\Gbb_m^n}(\{x_1, \dots, x_n\})$ of the Milnor symbol $\{x_1, \dots, x_n\} \in H^n_\Mcal(\Gbb_m^n, \Qbb(n))$. Since
	\begin{equation}
		\begin{aligned}
			\eta(x_1, \dots, x_n)|_{\Gamma \cap V_P^\reg} 
			= (-1)^{n-1} \log|x_n| \dfrac{d x_1}{x_1} \wedge \cdots \wedge \dfrac{d x_{n-1}}{x_{n-1}},
		\end{aligned}
	\end{equation}
	we arrive at the following result of Deninger (see [\cite{Den97}, Proposition 3.3]\label{Deningermethod}). If $\Gamma \subset V_P^\reg$ then
	\begin{equation}\label{Deningerformulae}
		\m (P) = \m (\tilde P) + \dfrac{(-1)^{n-1}}{(2\pi i)^{n-1}} \int_\Gamma \eta(x_1, \dots, x_n).
	\end{equation}
The boundary of $\Gamma$ is given by $\partial\Gamma = V_P \cap \{(x_1, \dots,  x_n) \in (\Cbb^\times)^n : |x_1| = \cdots = |x_n| = 1\}.$
	If $P$ does not vanish on $\Tbb^n$ (i.e., $\partial \Gamma = \varnothing$), then $\Gamma$ defines an element in $H_{k-1}(V_P^\reg, \Qbb)$, and we can write
	\begin{equation*}
		\m (P) = \m (\tilde P) + \dfrac{(-1)^{n-1}}{(2 \pi i)^{n-1}} \left<[\Gamma], \reg^{n,n}_{V_P^\reg} (\{x_1, \dots, x_n\})\right>,
	\end{equation*}
	where the pairing is $ \left< \cdot, \cdot\right> : H_{k-1}(V_P^\reg, \Qbb) \times H^{k-1}(V_P^\reg, \Rbb(n-1)) \to \Rbb(n-1).$ 
     Let us consider the case $\partial \Gamma \neq \varnothing$.  Maillot suggested to consider the following variety
	\begin{equation}\label{Maillot}
		W_P := V_P \cap V_{P^*},
	\end{equation}
	 where $P^* (x_1, \dots, x_n) :=  \bar P(x_1^{-1}, \dots, x_n^{-1})$. We call $W_P$ the \textit{Maillot variety}. From now on, we assume that $P$ has $\Qbb$-coefficients. Then
	$\partial \Gamma$ is contained in $W_P$. We introduce the definition of \textit{exact polynomials}.
	\begin{definition}[\textbf{Exact polynomials}]\label{a} Let $P \in \Qbb[x_1, \dots, x_n]$. We say that $P$ is exact if the class $\reg^{n,n}_{V_P^\reg}(\{x_1, \dots, x_n\})$ is trivial, i.e., the differential form $\eta(x_1, \dots, x_n)$ \eqref{Denform} is an exact form on $V_P^\reg$.
	\end{definition}

	We have the following lemma. 
	\begin{lm}\label{DenGon}
		We have $\eta(x_1, \dots, x_n) = -r_n(n)(x_1\wedge \dots \wedge x_n)$, where $r_n(n)(x_1 \wedge \dots \wedge x_n)$ is the differential form of Goncharov  \eqref{gonform}.
	\end{lm}
	
	\begin{proof} Let $m = \left[\frac{n-1}{2}\right]$. Recall that  
		\begin{align*}
			&r_n(n)(x_1, \dots, x_n) \\
			&= \mathrm{Alt}_n \left(\sum_{j=0}^{m} c_{j, n} \log |x_1| d \log|x_2| \wedge \cdots \wedge d \log|x_{2j+1}| \wedge d i \arg x_{2j+2} \wedge \cdots \wedge d i \arg x_n\right)\\
			&= - \sum_{\sigma \in \mathcal{S}_n} \sgn(\sigma)  \sum_{j=0}^{m} c_{j, n} \log |x_{\sigma_1}| d \log|x_{\sigma_2}| \wedge \cdots \wedge d \log|x_{\sigma_{2j+1}}| \wedge d i \arg x_{\sigma_{2j+2}} \wedge \cdots \wedge d i \arg x_{\sigma_n},
		\end{align*}
		where $c_{j, n} = \frac{1}{(2j+1)!(n-2j-1)!}$. We have
		\begin{equation}\label{sum}
			\begin{aligned}
				&\sum_{j=0}^{m} c_{j, n} \log |x_1| d \log|x_2| \wedge \cdots \wedge d \log|x_{2j+1}| \wedge d i \arg x_{2j+2} \wedge \cdots \wedge d i \arg x_n\\
				&= \sum_{j=0}^{m} c_{j, n} \varepsilon_1 (\partial \varepsilon_2 + \bar \partial \varepsilon_2) \wedge \cdots \wedge (\partial \varepsilon_{2j+1} + \bar \partial \epsilon_{2j+1}) \wedge (\partial \varepsilon_{2j+2} - \bar \partial \varepsilon_{2j+2}) \wedge \cdots \wedge (\partial \varepsilon_n - \bar \partial \varepsilon_n)\\
				&= c_{0,n} \  \varepsilon_1 \ (\partial \varepsilon_2 - \bar \partial \varepsilon_2) \wedge \cdots \wedge (\partial \varepsilon_n - \bar \partial \varepsilon_n) \\
				&\hspace{2cm} \ + c_{1,n} \  \varepsilon_1 \ (\partial \varepsilon_2 + \bar \partial \varepsilon_2) \wedge (\partial \varepsilon_3 + \bar \partial \varepsilon_3) \wedge (\partial \varepsilon_4 - \bar \partial \varepsilon_4) \wedge \cdots \wedge (\partial \varepsilon_n - \bar \partial \varepsilon_n)  \\
				&\hspace{2cm} \ +  c_{2,n} \  \varepsilon_1 \ (\partial \varepsilon_2 + \bar \partial \varepsilon_2) \wedge \cdots \wedge (\partial \varepsilon_5 + \bar \partial \varepsilon_5) \wedge (\partial \varepsilon_6 - \bar \partial \varepsilon_6) \wedge \cdots \wedge (\partial \varepsilon_n - \bar \partial \varepsilon_n) \\
				&\hspace{2cm} \quad \vdots\\
				&\hspace{2cm} \ +  c_{m,n} \  \varepsilon_1 \ (\partial \varepsilon_2 + \bar \partial \varepsilon_2) \wedge (\partial \varepsilon_3 + \bar \partial \varepsilon_3)  \wedge \cdots \wedge (\partial \varepsilon_{n-1} + \bar \partial \varepsilon_{n-1}) \wedge (\partial \varepsilon_n - \bar \partial \varepsilon_n) \\
				&=  \left(\sum_{j=0}^m c_{j, n}\right) \varepsilon_1 \ \partial \varepsilon_2 \wedge \cdots \wedge \partial \varepsilon_n \\
				& + \sum_{k=0}^{m-1} \left[\left(-c_{0,n} - \cdots - c_{k,n} + c_{k+1, n} +\cdots + c_{m, n}\right) \varepsilon_1 \ \partial \varepsilon_2  \wedge \partial \varepsilon_3 \wedge \cdots \wedge \bar \partial \varepsilon_{2k+2} \wedge \partial \varepsilon_{2k+3} \wedge \cdots \wedge \partial \varepsilon_n \right. \\
				& \left. + \left(-c_{0,n} - \cdots - c_{k,n} + c_{k+1, n} +\cdots + c_{m, n}\right) \varepsilon_1 \ \partial \varepsilon_2  \wedge \partial \varepsilon_3 \wedge \cdots \wedge  \partial \varepsilon_{2k+2} \wedge \bar \partial \varepsilon_{2k+3} \wedge \cdots \wedge \partial \varepsilon_n \right] \\
				&+ \left(-c_{0,n}  -\cdots - c_{m, n}\right) \varepsilon_1 \ \partial \varepsilon_2  \wedge \partial \varepsilon_3 \wedge \cdots \wedge \partial \varepsilon_{n-1}  \wedge \bar \partial \varepsilon_n\\
				&\ \vdots\\
				& +  \left(\sum_{j=0}^m (-1)^{n-2j-1} c_{j, n}\right) \varepsilon_1 \ \bar \partial \varepsilon_2 \wedge \cdots \wedge \bar \partial \varepsilon_n.
			\end{aligned}
		\end{equation}
		Hence the coefficients of $\varepsilon_1 \ \partial \varepsilon_2 \wedge \cdots \wedge \partial \varepsilon_n$ and $\varepsilon_1 \ \bar \partial \varepsilon_2 \wedge \cdots \wedge \bar \partial \varepsilon_n$ in this sum are $\sum_{j=0}^m c_{j, n} = \frac{2^{n-1}}{n!}$ and $$\sum_{j=0}^m (-1)^{n-2j-1} c_{j, n} = (-1)^{n-1} \sum_{j=0}^m c_{j, n} = (-1)^{n-1} \frac{2^{n-1}}{n!},$$
		respectively. To compute the coefficient of 
		$\sum_{\sigma \in \mathcal{S}_n} \mathrm{sgn}(\sigma) \varepsilon_{\sigma_1} \ \bar \partial \varepsilon_{\sigma_2}  \wedge \partial \varepsilon_{\sigma_3} \wedge \cdots \wedge \partial \varepsilon_{\sigma_n}$
		in $r_n(n)(x_1 \wedge \cdots \wedge x_n)$, let us see how the coefficients of 
		\begin{equation}\label{2k+2}
			\varepsilon_1 \ \partial \varepsilon_2  \wedge \partial \varepsilon_3 \wedge \cdots \wedge \bar \partial \varepsilon_{2k+2} \wedge \partial \varepsilon_{2k+3} \wedge \cdots \wedge \partial \varepsilon_n,
		\end{equation}
		and 
		\begin{equation}\label{2k+3}
			\varepsilon_1 \ \partial \varepsilon_2  \wedge \partial \varepsilon_3 \wedge \cdots \wedge  \partial \varepsilon_{2k+2} \wedge \bar \partial \varepsilon_{2k+3} \wedge \cdots \wedge \partial \varepsilon_n
		\end{equation}
		change after we transform them into the form 
		\begin{equation}\label{eps2}
			\varepsilon_1 \ \bar \partial \varepsilon_2  \wedge \partial \varepsilon_3 \wedge \cdots \wedge \partial \varepsilon_n.
		\end{equation}
		To transform \eqref{2k+2} into \eqref{eps2}, we first  permute $\bar \partial \varepsilon_{2k+2}$ $2k$ times to get
		\begin{equation}\label{2k+2'}
			\varepsilon_1 \ \bar \partial \varepsilon_{2k+2}  \wedge \partial \varepsilon_2 \wedge \cdots \wedge  \partial \varepsilon_{2k+1} \wedge \partial \varepsilon_{2k+3} \wedge \cdots \wedge \partial \varepsilon_n.
		\end{equation}
		This transformation does not change the sign of the coefficient. Now to transform \eqref{2k+2'} to \eqref{eps2}, we apply the permutation 
		\begin{equation}
			\tau = (2k+2 \quad 2 \quad 3 \quad \cdots\quad 2k+1),
		\end{equation}
		whose length is $2k+1$, hence $\mathrm{sgn}(\tau) = (-1)^{2k+1-1} = 1$. More precisely, we have
		\begin{equation*}
			\begin{aligned}
				&\sum_{\sigma\in \mathcal{S}_n} \mathrm{sgn}(\sigma) \ \varepsilon_{\sigma_1} \ \partial \varepsilon_{\sigma_2} \wedge \partial \varepsilon_{\sigma_3} \wedge \cdots \wedge \bar \partial \varepsilon_{\sigma_{2k+2}}  \wedge \partial \varepsilon_{\sigma_{2k+3}} \wedge \cdots \wedge \partial \varepsilon_{\sigma_n} \\
				&=  \sum_{\sigma \in \mathcal{S}_n} \mathrm{sgn}(\sigma)  \ \varepsilon_{\sigma_1} \  \bar \partial \varepsilon_{\sigma_{2k+2}} \wedge \partial \varepsilon_{\sigma_2}  \wedge \partial \varepsilon_{\sigma_3} \wedge \cdots \wedge \partial \varepsilon_{\sigma_{2k+1}} \wedge \partial \varepsilon_{\sigma_{2k+3}} \wedge \cdots \wedge \partial \varepsilon_{\sigma_n}\\
				&=  \sum_{\sigma \in \mathcal{S}_n} \mathrm{sgn}(\sigma \tau) \ \varepsilon_{(\sigma\tau)_1} \  \bar \partial \varepsilon_{(\sigma \tau)_{2k+2}} \wedge \partial \varepsilon_{(\sigma \tau)_2}  \wedge \partial \varepsilon_{(\sigma \tau)_3} \wedge \cdots \wedge \partial \varepsilon_{(\sigma\tau)_{2k+1}} \wedge \partial \varepsilon_{(\sigma\tau)_{2k+3}} \wedge \cdots \wedge \partial \varepsilon_{(\sigma\tau)_n}\\
				&= \sum_{\sigma \in \mathcal{S}_n}  \mathrm{sgn}(\sigma) \  \varepsilon_{\sigma_1} \ \bar \partial \varepsilon_{\sigma_2}  \wedge \partial \varepsilon_{\sigma_3}  \wedge \cdots \wedge \partial \varepsilon_{\sigma_n}.
			\end{aligned}
		\end{equation*}
		We apply the same procedure to transform \eqref{2k+3} to \eqref{eps2}, and we also get  
		\begin{equation*}
			\sum_{\sigma\in \mathcal{S}_n} \mathrm{sgn}(\sigma) 
			\ \varepsilon_{\sigma_1} \ \partial \varepsilon_{\sigma_2}  \wedge \partial \varepsilon_{\sigma_3} \wedge \cdots \wedge  \partial \varepsilon_{\sigma_{2k+2}} \wedge \bar \partial \varepsilon_{\sigma_{2k+3}} \wedge \cdots \wedge \partial \varepsilon_{\sigma_n}  = \sum_{\sigma\in \mathcal{S}_n} \mathrm{sgn}(\sigma) 
			\ \varepsilon_{\sigma_1} \ \bar \partial \varepsilon_{\sigma_2}  \wedge \partial \varepsilon_{\sigma_3} \wedge \cdots \wedge \partial \varepsilon_{\sigma_n}.
		\end{equation*}
		Thus, the coefficient of $\sum_{\sigma\in \mathcal{S}_n} \mathrm{sgn}(\sigma) 
		\ \varepsilon_1 \ \bar \partial \varepsilon_2  \wedge \partial \varepsilon_3 \wedge \cdots \wedge \partial \varepsilon_n$ in $r_n(n)(x_1 \wedge \cdots x_n)$ is 
		\begin{equation*}
			\begin{aligned}
				2 \sum_{k=0}^{m-1} \left(-c_{0,n} - \cdots - c_{k,n} + c_{k+1, n} +\cdots + c_{m, n}\right) + (-c_{0, n} - \cdots - c_{m,n}).
			\end{aligned}
		\end{equation*}
		which is equal to $-\frac{2^{n-1}}{n!}$. Similarly, one can prove that the coefficient of 
		\begin{equation*}
			\sum_{\sigma \in \mathcal{S}_n} \mathrm{sgn} (\sigma) \varepsilon_{\sigma_1} \ \bar \partial \varepsilon_{\sigma_2} \wedge \bar \partial \varepsilon_{\sigma_3} \wedge \cdots \wedge \bar \partial \varepsilon_{\sigma_j} \wedge \partial\varepsilon_{\sigma_{j+1}} \cdots \wedge \partial \varepsilon_{\sigma_n})
		\end{equation*}
		is $(-1)^{j-1} \frac{2^{n-1}}{n!}$.
	\end{proof}

	\begin{ex}
		For $n=2$, we have
		\begin{equation*}
			\begin{aligned}
				r_2(2)(x_1 \wedge x_2) 
				= \mathrm{Alt}_2 (\log|x_1| d i \arg x_2)
				= - \log|x_1| d i \arg x_2 + \log |x_2| d i \arg |x_1|
				= - \eta(x_1, x_2).
			\end{aligned}
		\end{equation*}
	\end{ex}

	\begin{ex}
		Consider the case $n =3$. We have
		\begin{equation*}
			\begin{aligned}
				&r_3(3) (x_1 \wedge x_2 \wedge x_3) \\
				&= \mathrm{Alt}_3 \left(\dfrac{1}{2} \log |x_1| di \arg x_2 \wedge d i \arg x_3  + \dfrac{1}{6} 
				\log |x_1| d \log |x_2| \wedge d \log|x_3| \right)\\
				&= \mathrm{Alt}_3 \left( \dfrac{1}{2} \log|x_1| (\partial \log|x_2| - \bar \partial \log |x_2|) \wedge (\partial \log|x_3| - \bar \partial \log |x_3|) \right.\\
				&\ + \left. \dfrac{1}{6} \log|x_1| (\partial \log|x_2| + \bar \partial \log|x_2|) \wedge (\partial \log|x_3| + \bar \partial \log |x_3|) \right)\\
				&= \dfrac{2}{3} \mathrm{Alt}_3\left( \log|x_1| \partial \log|x_2| \partial \log |x_3|\right) + \dfrac{2}{3} \mathrm{Alt}_3\left(\log|x_1| \bar \partial \log|x_2| \bar \partial \log |x_3|\right)\\
				&\ -  \dfrac{1}{2} \mathrm{Alt}_3\left( \log|x_1|\bar  \partial \log|x_2| \partial \log |x_3|\right) -  \dfrac{1}{2} \mathrm{Alt}_3\left( \log|x_1|  \partial \log|x_2| \bar \partial \log |x_3|\right) \\
				&\ + \dfrac{1}{6} \mathrm{Alt}_3\left( \log|x_1|\bar  \partial \log|x_2| \partial \log |x_3|\right) +  \dfrac{1}{6} \mathrm{Alt}_3\left( \log|x_1|  \partial \log|x_2| \bar \partial \log |x_3|\right)\\
				&= \dfrac{2}{3} \mathrm{Alt}_3\left( \log|x_1| \partial \log|x_2| \partial \log |x_3|\right) + \dfrac{2}{3} \mathrm{Alt}_3\left(\log|x_1| \bar \partial \log|x_2| \bar \partial \log |x_3|\right) -  \dfrac{2}{3} \mathrm{Alt}_3\left( \log|x_1| \bar \partial \log|x_2| \partial \log |x_3|\right)\\
				&= \dfrac{2}{3}\mathrm{Alt}_3 \left( \log|x_1| \partial \log|x_2| \partial \log |x_3| - \log|x_1| \bar \partial \log|x_2| \partial \log |x_3| +\log|x_1| \bar \partial \log|x_2| \bar \partial \log |x_3| \right)\\
				&= - \eta(x_1, x_2, x_3).
			\end{aligned}
		\end{equation*}
	\end{ex}

	\begin{ex}
		We consider the case $n=4$. We have 
		\begin{equation*}
			\begin{aligned}
				&r_4(3)(x_1 \wedge x_2 \wedge x_3 \wedge x_4) 
				= \mathrm{Alt}_4 \left(\dfrac{1}{6} 
				\log |x_1| d i \arg x_2 \wedge d i \arg x_3 \wedge d i \arg x_4 \right.\\
				& \left. \hspace{5cm} + \dfrac{1}{6} 
				\log |x_1| d \log |x_2| \wedge d \log|x_3| \wedge d i \arg |x_4| \right)\\
				&= - \dfrac{1}{6} \sum_{\sigma \in \Scal_4} \sgn(\sigma) \log |x_{\sigma(1)}| d i \arg x_{\sigma(2)} \wedge d i \arg x_{\sigma(3)} \wedge d i \arg x_{\sigma(4)}\\
				&\quad - \dfrac{1}{6} \sum_{\sigma \in \Scal_4} \sgn(\sigma)  \log |x_{\sigma(1)}| d \log |x_{\sigma(2)}| \wedge d \log|x_{\sigma(3)}| \wedge d d i \arg |x_{\sigma(4)}| \\
				&=- \dfrac{1}{48} \sum_{\sigma \in \Scal_4} \sgn(\sigma) \log |x_{\sigma(1)}| \left(\dfrac{d x_{\sigma(2)}}{x_{\sigma(2)}} - \dfrac{d \bar x_{\sigma(2)}}{\bar x_{\sigma(2)}}\right) \wedge \left(\dfrac{d x_{\sigma(3)}}{x_{\sigma(3)}} - \dfrac{d \bar x_{\sigma(3)}}{\bar x_{\sigma(3)}}\right) \wedge \left(\dfrac{d x_{\sigma(4)}}{x_{\sigma(4)}} - \dfrac{d \bar x_{\sigma(4)}}{\bar x_{\sigma(4)}}\right)\\
				&\quad - \dfrac{1}{48} \sum_{\sigma \in \Scal_4} \sgn(\sigma) \log |x_{\sigma(1)}| \left(\dfrac{d x_{\sigma(2)}}{x_{\sigma(2)}} + \dfrac{d \bar x_{\sigma(2)}}{\bar x_{\sigma(2)}}\right) \wedge \left(\dfrac{d x_{\sigma(3)}}{x_{\sigma(3)}} + \dfrac{d \bar x_{\sigma(3)}}{\bar x_{\sigma(3)}}\right) \wedge \left(\dfrac{d x_{\sigma(4)}}{x_{\sigma(4)}} -  \dfrac{d \bar x_{\sigma(4)}}{\bar x_{\sigma(4)}}\right)\\
				&=- \dfrac{1}{24} \sum_{\sigma\in \Scal_4} \sgn(\sigma) \log |x_{\sigma(1)}| \dfrac{d x_{\sigma(2)}}{x_{\sigma(2)}} \wedge \dfrac{d x_{\sigma(3)}}{x_{\sigma(3)}} \wedge \dfrac{d x_{\sigma(4)}}{x_{\sigma(4)}} + \dfrac{1}{24} \sum_{\sigma\in \Scal_4} \sgn(\sigma) \log |x_{\sigma(1)}| \dfrac{d \bar x_{\sigma(2)}}{x_{\sigma(2)}} \wedge \dfrac{d x_{\sigma(3)}}{x_{\sigma(3)}} \wedge \dfrac{d x_{\sigma(4)}}{x_{\sigma(4)}}\\
				&\quad - \dfrac{1}{24} \sum_{\sigma\in \Scal_4} \sgn(\sigma) \log |x_{\sigma(1)}| \dfrac{d \bar x_{\sigma(2)}}{\bar x_{\sigma(2)}} \wedge \dfrac{d \bar x_{\sigma(3)}}{\bar x_{\sigma(3)}} \wedge \dfrac{d x_{\sigma(4)}}{x_{\sigma(4)}} + \dfrac{1}{24} \sum_{\sigma\in \Scal_4} \sgn(\sigma) \log |x_{\sigma(1)}| \dfrac{d \bar x_{\sigma(2)}}{\bar x_{\sigma(2)}} \wedge \dfrac{d \bar x_{\sigma(3)}}{\bar x_{\sigma(3)}} \wedge \dfrac{d \bar x_{\sigma(4)}}{\bar x_{\sigma(4)}}\\
				&= - \dfrac{1}{3} \sum_{\sigma\in \Scal_4} \sgn(\sigma) \log |x_{\sigma(1)}| \partial \log |x_{\sigma(2)}| \wedge \partial \log |x_{\sigma(3)}| \wedge \partial \log |x_{\sigma(4)}|\\
				&\quad + \dfrac{1}{3} \sum_{\sigma\in \Scal_4} \sgn(\sigma) \log |x_{\sigma(1)}| \bar \partial \log |x_{\sigma(2)}| \wedge \partial \log |x_{\sigma(3)}| \wedge \partial \log |x_{\sigma(4)}|\\
				&\quad - \dfrac{1}{3} \sum_{\sigma\in \Scal_4} \sgn(\sigma) \log |x_{\sigma(1)}| \bar \partial \log |x_{\sigma(2)}| \wedge \bar \partial \log |x_{\sigma(3)}| \wedge \partial \log |x_{\sigma(4)}|\\
				&\quad + \dfrac{1}{3} \sum_{\sigma\in \Scal_4} \sgn(\sigma) \log |x_{\sigma(1)}| \bar \partial \log |x_{\sigma(2)}| \wedge \bar \partial \log |x_{\sigma(3)}| \wedge \bar \partial \log |x_{\sigma(4)}|\\
				&= - \eta(x_1, x_2, x_3, x_4).
			\end{aligned}
		\end{equation*}
	\end{ex}
	
	Now we assume the following decomposition on $\bigwedge^n \Qbb(V_P)^\times_\Qbb$, which generalizes a condition considered by Lalín in the three-variable case (see  \cite{Lal15}),
	
	\begin{equation}\label{lalincond}
		x_1 \wedge \cdots \wedge x_n = \sum_j c_j\   f_j \wedge (1-f_j)\wedge \bigwedge_{k=1}^{n-2} g_{j,k},
	\end{equation}
	for some $c_j \in \Qbb^\times$ and some functions $f_j \in \Qbb(V_P)^\times \setminus \{1\}$ and $g_{j, k} \in \Qbb(V_P)^\times$ for $k=1, \dots, n-2$.  Then by Lemma \ref{DenGon}, we have 
	\begin{equation*}\label{primitive}
		\begin{aligned}
			\eta(x_1, \dots, x_n)|_{V_P^\reg} &= - r_n(n)(x_1 \wedge \dots \wedge x_n)|_{V_P^\reg}\\
			&= - \sum_j c_j \ r_n(n) (f_j \wedge (1-f_j) \wedge \bigwedge_{k=1}^{n-2} g_{j,k})\\
			&= \sum_j c_j \ r_n(n) ((1-f_j) \wedge f_j \wedge \bigwedge_{k=1}^{n-2} g_{j,k})\\
			&= d \left(\sum_j c_j r_n(n-1) (\{f_j\}_2 \otimes \bigwedge_{k=1}^{n-2} g_{j,k}) \right), 
		\end{aligned}
	\end{equation*}
	where the last equality follows from the commutative diagram \eqref{10.5.1} of Goncharov. Hence $\eta(x_1, \dots, x_n)$ is an exact form on $V_P^\reg$, so that $P$ is an exact polynomial. We consider the following involution on $(\Cbb^\times)^n$
	\begin{equation}\label{tau}
		\tau : (\Cbb^\times)^n \to (\Cbb^\times)^n, \quad (x_1, \dots, x_n) \mapsto (1/x_1, \dots, 1/x_n),
	\end{equation}
	which is defined over $\Qbb$.  Its restriction on $W_P$ is an isomorphism. 
	
	\begin{definition}\label{310} Denote by  $F = \Qbb(W_P)$.  We define the following elements in $B_2(F)  \otimes \bigwedge^{n-2} F^\times_\Qbb$.
		\begin{equation}\label{lambda}
			\xi := \sum_j c_j \{f_j\}_2 \otimes \bigwedge_{k=1}^{n-2} g_{j,k}, \qquad \xi^* := \sum_j c_j \{f_j \circ \tau\}_2 \otimes \bigwedge_{k=1}^{n-2} (g_{j,k} \circ \tau), \qquad \lambda := \dfrac{1}{2}\left(\xi + (-1)^{n-1} \xi^*\right),
		\end{equation}
		where $f_j, g_{j, k}$ in the decomposition \eqref{lalincond}. Let us consider the following closed subschemes of $V_P$ and $V_{P^*}$, respectively \begin{equation}\label{233}
			Z_1 = \{\text{zeros and poles of  } f_j, 1-f_j, g_{j, k}   \text{ as functions on } V_{P} \text{ for all } j, k\},
		\end{equation} \begin{equation}
			Z_2 = \{\text{zeros and poles of  } f_j \circ \tau, 1-f_j \circ \tau, g_{j, k} \circ \tau \text{ as functions on } V_{P^*} \text{ for all } j, k\}.
		\end{equation}
		We define the following differential $(n-2)$-forms on $V_P^\reg \setminus Z_1$ and $V_{P^*}^\reg \setminus Z_2$, respectively.
		\begin{equation}\label{Z}
			\rho(\xi) := \sum_j c_j \rho(f_j, g_{j, 1}, \dots, g_{j,n-2}), \quad \rho(\xi^*) := \sum_j c_j \rho(f_j \circ \tau, g_{j, 1} \circ \tau, \dots, g_{j, n-2} \circ \tau),
		\end{equation}
		where $\rho(f, g_1, \dots, g_{n-2})$ is the differential $(n-2)$-form representing $r_n(n-1)(\{f\}_2 \otimes \bigwedge_{k=1}^{n-2} g_{j, k})$ as mentioned in $\eqref{gonform}$. We consider the following closed subscheme of $W_P$
        \begin{equation}
            Z = \{\text{zeros and poles of  } f_j, 1-f_j, g_{j, k}, f_j \circ \tau, 1-f_j \circ \tau, g_{j, k} \circ \tau   \text{ as functions on } W_P \text{ for all } j, k\}.
        \end{equation}
	We set \begin{equation}\label{Y}
			Y := W_P^\reg \setminus Z,
		\end{equation}
		and consider the following differential $(n-2)$-form defined on $Y(\Cbb)$
		\begin{equation}
			\rho(\lambda) := \dfrac{1}{2}\left(\rho(\xi)|_{Y(\Cbb)} + (-1)^{n-1}\rho(\xi^*)|_{Y(\Cbb)} \right).
		\end{equation}
	\end{definition}
	
	We have the following result.
	
	\begin{lm}\label{Hn-1Fn}
		The element $\lambda$ \eqref{lambda} defines a class in $H^{n-1}(\Gamma(F, n))$, where $\Gamma(F, n)$ is the weight $n$ polylogarithmic complex of Goncharov \eqref{weightnGoncomplex}. 
	\end{lm}
	\begin{proof}
		We recall the weight $n$ polylogarithmic complex of Goncharov 
		\begin{equation}
			\Bscr_n(F) \to \Bscr_{n-1}(F) \otimes F^\times_\Qbb \to \Bscr_{n-2}(F) \otimes \bigwedge^2 F^\times_\Qbb \to \cdots \to \Bscr_2(F) \otimes \bigwedge^{n-2} F^\times_\Qbb \to \bigwedge^n F^\times_\Qbb,
		\end{equation}
		where the differentials map at degree $n-1$ is defined by
		\begin{equation*}
			\alpha_n(n-1) : \{x\}_2 \otimes y_1 \wedge \dots \wedge y_{n-2} \mapsto (1-x)\wedge x \wedge y_1 \wedge \dots \wedge y_{n-2}.
		\end{equation*}
		We have $\alpha_n(n-1) (\xi) = \sum_j c_j (1-f_j)\wedge f_j \wedge \bigwedge^{n-2}_{k=1} g_{j, k} = - x_1 \wedge \dots \wedge x_n$
		and
		\begin{align*}
			\alpha_n(n-1) (\xi^*) &= \sum_j c_j (1-f_j \circ \tau) \wedge (f_j \circ \tau) \wedge \bigwedge^{n-2}_{k=1} (g_{j, k} \circ \tau)\\
			&= \sum_j c_j \tau^*((1-f_j)\wedge f_j \wedge \bigwedge^{n-2}_{k=1} g_{j, k})\\
			&= \tau^* (-x_1 \wedge \dots \wedge x_n)\\
			&= (-1)^{n-1} x_1 \wedge \dots \wedge x_n.
		\end{align*}
	\end{proof}
	
	\begin{lm}\label{7.2.7}
		The differential form $\rho(\lambda)$ is closed and invariant under the action of the involution $F_{\dR}$. Then it defines a  class in Deligne cohomology $H^{n-1}_\Dcal(Y_\Rbb, \Rbb(n)) \simeq H^{n-2}(Y(\Cbb), \Rbb(n-1))^+,$ where $Y(\Cbb)$ is the set of complex points of $Y \times_\Qbb \Cbb$.
	\end{lm}
	
	\begin{proof}
		The proof for the first statement  is similar to the proof of Lemma \ref{Hn-1Fn}. Indeed, we have $$d \rho (\xi) = \eta(x_1, \cdots, x_n)|_{V_P^\reg},$$
		and
		$$d \rho(\xi^*) = d (\tau^* \rho(\xi)) = \tau^* d \rho(\xi) = \tau^* \eta(x_1, \dots, x_n) = (-1)^n \eta(x_1, \dots, x_n)|_{V_P^\reg}.$$
		Therefore, $d\rho(\lambda) = d \rho(\xi) + (-1)^{n-1} d \rho(\xi^*) = 0$. Recall that the differential form $\rho(f, g_1, \dots, g_{n-2})$ is 
		\begin{align*}
			&i D(f) \sum_{p=0}^{[\frac{n-2}{2}]} \dfrac{1}{(2p+1)!(n-2p-2)!} \mathrm{Alt}_{n-2} \left( \bigwedge_{k=1}^{2p} d \log |g_k| \wedge \bigwedge_{\ell=2p+1}^{n-2} d i \arg g_\ell \right)\\
			&+ \sum_{m=1}^{[\frac{n-1}{2}]} \dfrac{1}{(2m+1)(2m-1)!(n-2m-1)!}\theta(1-f, f) \mathrm{Alt}_{n-2} \left(\log |g_1| \bigwedge_{k=2}^{2m-1}d \log|g_k| \wedge\bigwedge_{\ell = 2m}^{n-2} di \arg g_\ell\right),
		\end{align*}
		where $\theta(f, g) :=  \log|f| d \log |g| - \log |g| d \log |f|$. Now we show that $\rho(f, g_1, \dots, g_{n-2})$ is invariant under the action of the involution $F_{\dR}$. Recall that for a differential form 
		$\omega$, $F_{\dR} (\omega) := c^*(\bar \omega)$, where $c$ is complex conjugation. We have that  $\bar \rho(f, g_1, \dots, g_{n-2})$ is given by
		\begin{align*}
			&-D(f) \sum_{p=0}^{[\frac{n-2}{2}]} \dfrac{1}{(2p+1)!(n-2p-2)!} \mathrm{Alt}_{n-2} \left( \bigwedge_{k=1}^{2p} d \log |g_k| \wedge \bigwedge_{\ell=2p+1}^{n-2} d \arg g_\ell \right)\\
			&-i \sum_{m=1}^{[\frac{n-1}{2}]} \dfrac{1}{(2m+1)(2m-1)!(n-2m-1)!}\theta(1-f, f) \mathrm{Alt}_{n-2} \left(\log |g_1| \bigwedge_{k=2}^{2m-1}d \log|g_k| \wedge\bigwedge_{\ell = 2m}^{n-2} d \arg g_\ell\right).
		\end{align*}
		Let $z \in Y(\Cbb)$. The image of $\bar \rho(f, g_1, \dots, g_{n-2})(z)$ under the action of complex conjugation $c$ is 
		\begin{align*}
			&-D(f(\bar z)) \sum_{p=0}^{[\frac{n-2}{2}]} \dfrac{1}{(2p+1)!(n-2p-2)!} \mathrm{Alt}_{n-2} \left( \bigwedge_{k=1}^{2p} d \log |g_k (\bar z)| \wedge \bigwedge_{\ell=2p+1}^{n-2} d \arg g_\ell (\bar z) \right)\\
			&-i \sum_{m=1}^{[\frac{n-1}{2}]} \dfrac{1}{(2m+1)(2m-1)!(n-2m-1)!}\theta(1-f(\bar z), f(\bar z))\\
			&\hspace{7cm} \mathrm{Alt}_{n-2} \left(\log |g_1(\bar z)| \bigwedge_{k=2}^{2m-1}d \log|g_k(\bar z)| \wedge\bigwedge_{\ell = 2m}^{n-2} d \arg g_\ell (\bar z)\right).
		\end{align*}
		As $f_j, g_j$ have rational coefficients, this equals
		\begin{equation*}
			\begin{aligned}
				&-D(\overline{f(z)}) \sum_{p=0}^{[\frac{n-2}{2}]} \dfrac{1}{(2p+1)!(n-2p-2)!} \mathrm{Alt}_{n-2} \left( \bigwedge_{k=1}^{2p} d \log |\overline{g_k (z)}| \wedge \bigwedge_{\ell=2p+1}^{n-2} d \arg \overline{g_\ell (z)} \right)\\
				&-i \sum_{m=1}^{[\frac{n-1}{2}]} \dfrac{1}{(2m+1)(2m-1)!(n-2m-1)!}\theta(1-\overline{f(z)}, \overline{f(z)})\\
				&\hspace{7cm} \mathrm{Alt}_{n-2} \left(\log |\overline{g_1(z)}| \bigwedge_{k=2}^{2m-1}d \log| \overline{g_k(z)}| \wedge\bigwedge_{\ell = 2m}^{n-2} d \arg \overline{g_\ell (z)}\right),
			\end{aligned}  
		\end{equation*}
		which is the same as $\rho(f, g_1,\dots, g_{n-2})$ since $D(\bar z) = - D(z)$, $|\bar z| = |z|$, $\arg \bar z  = - \arg z$, $|1-\bar z| = |1-z|$ for $z \in \Cbb$.  
	\end{proof}

	\begin{thm}\label{7.2.8} Assume that $\Gamma\subset V_P^\reg$ and that $\partial \Gamma \subset Y(\Cbb)$. Then $\partial \Gamma$ defines an element in the singular homology group $H_{n-2}(Y(\Cbb), \Zbb)^{(-1)^{n-1}}$, where $``+/-"$ denotes the invariant/anti-invariant part of $H_{n-2}(Y(\Cbb), \Zbb)$ under the action induced by the complex conjugation. Under condition \eqref{lalincond}, we have 
		\begin{equation}
			\m(P) - \m(\tilde P) = \dfrac{(-1)^{n-1}}{(2\pi i)^{n-1}} \int_{\partial \Gamma} \rho(\xi).
		\end{equation}
		Moreover, the Mahler measure can be written as a pairing in Deligne cohomology of $Y_\Rbb$
		\begin{equation}\label{pair}
			\m (P) - \m (\tilde P) =  \dfrac{(-1)^{n-1}}{(2\pi i)^{n-1}} \left<[\partial \Gamma], [\rho(\lambda)]\right>_Y,
		\end{equation}
		where the pairing is given by
		\begin{equation}
			\left<\cdot, \cdot\right>_Y : H_{n-2}(Y(\Cbb), \Zbb)^{(-1)^{n-1}} \times H^{n-2}(Y(\Cbb), \Rbb(n-1))^+ \to \Rbb(n-1).
		\end{equation}
	\end{thm}
	
	\begin{proof}
		Since $\partial \Gamma$ is contained in $Y(\Cbb)$,  we have the following sequence
		\begin{center}
			\begin{tikzcd}[cramped,row sep=small]
				H_{n-1}(V_P^\reg, \partial \Gamma, \Zbb) \arrow[r] & H_{n-2}(\partial \Gamma, \Zbb) \arrow[r] &  H_{n-2}(Y(\Cbb), \Zbb) \\
				{[\Gamma]} \arrow[r, mapsto, shorten=5mm] & {[\partial \Gamma]} \arrow[r, mapsto, shorten=5mm] & {[\partial \Gamma]}.
			\end{tikzcd}
		\end{center}
		Therefore, $\partial \Gamma$ defines an element in $H_{n-2}(Y(\Cbb), \Zbb)$.  Now we show that $\partial \Gamma$ is fixed under the action of complex conjugation. Notice that the action of complex conjugation on $\partial\Gamma$ is nothing but the action of $\tau$ \eqref{tau}. We have $\tau(\partial \Gamma) = \partial \Gamma$ as a set. We claim that $\tau$ preserves (reverses) the orientation of $\partial \Gamma$ if $n$ is odd (even), respectively. In fact, the orientation of the boundary $\partial \Gamma$ is induced by the orientation of $\Gamma$ and the orientation of $\Gamma$ comes from the orientation of the torus $\Tbb^{n-1} = (\mathbb{S}^1)^{n-1}$.  The claim follows from the fact that $\tau$ reverses the orientation of the unit circle $\mathbb{S}^1$. Therefore, $\partial \Gamma$ is invariant (anti-invariant) under the complex conjugation when $n$ is odd (even), respectively. 
		
		By condition \eqref{lalincond}, we have $\eta(x_1, \dots, x_n)|_{V_P^\reg} = d \rho(\xi)$. Applying Stokes' theorem to Deninger's formula \eqref{Deningerformulae}, we get 
		\begin{equation}\label{1form}
			\m(P) - \m(\tilde P) = \dfrac{(-1)^{n-1}}{(2\pi i)^{n-1}} \int_{\partial \Gamma} \rho(\xi)|_{Y(\Cbb)}.
		\end{equation}
		We have
		\begin{equation*}
			\int_{\partial \Gamma} \rho(\xi)|_{Y(\Cbb)} = \int_{\tau_*(\partial \Gamma)} \tau^*(\rho(\xi))|_{Y(\Cbb)} = (-1)^{n-1}
			\int_{\partial \Gamma} \rho(\xi^*)|_{Y(\Cbb)},
		\end{equation*}
		where the last equality is because $\tau$ fixes $\partial \Gamma$ as a set and $\tau$ preserves (inverses) the orientation of $\partial \Gamma$ if $n$ is odd (even), respectively. Hence 
		\begin{equation}
			\m(P)  - \m(\tilde P) = \dfrac{(-1)^{n-1}}{(2\pi i)^{n-1}} \int_{\partial \Gamma} \rho(\xi)|_{Y(\Cbb)} = \dfrac{(-1)^{n-1}}{(2\pi i)^{n-1}} \int_{\partial \Gamma} \rho(\lambda). 
		\end{equation}
		
	\end{proof}

	\section{The Mahler measure of four-variable exact polynomials\\and special $L$-values of $K3$ surfaces}\label{4varsgeneral}
	
	In this section, we study the relationship between the Mahler measure of four-variable exact polynomials and special $L$-values of $K3$ surfaces. In Section \ref{section 1.3}, using Goncharov's polylogarithmic complex, we construct an element in the motivic cohomology of a smooth compactification of the Maillot variety and relate it to the Mahler measure under some conjectures.  In Section \ref{section K3}, we study the motivic cohomology and Deligne-Beilinson cohomology of singular $K3$ surfaces, and then restate the Beilinson conjectures.
	
	\subsection{Mahler measures of exact polynomials in four variables}\label{section 1.3} While Section \ref{main1} dealt with polynomials in arbitrarily many variables $n$, we now turn our attention to the specific case $n = 4$. Let $X$ be a regular surface over a field $k$ of characteristic 0, and $F$ be the function field of $X$. Let $X^i$ be the set of points of codimension $i$ of $X$. Denote by $k(x)$ the residue field and $\pi_x$ a uniformizer at $x$. Let $G_x := \Gal(\overline{k(x)}/k(x))$. Goncharov showed that there is a canonical map  $H : \bigwedge^4 F^\times_\Qbb \to \oplus_{x\in X^2} \Bscr_2(\overline{k(x)})^{G_x}$
	satisfying that $H(k^\times \wedge \bigwedge^3 F^\times) = 0$ and $\alpha_2(1) \circ H = \partial^{3,3} \circ \partial^{4, 4}$, then constructed  a complex $\Gamma(X, 4)$  with the differentials $\alpha + \partial + H$  (see \cite[Theorem 8.1]{GR22}\footnote{There is a misprint in \cite[Theorem 8.1]{GR22}, the horizontal degree of the diagram should be counted from 0 to 2 from top to bottom.})
	\begin{equation}\label{bicomplex4}
		\begin{tikzcd}
			\Bscr_4(F) \arrow[r,"\alpha_4(1)"] & \Bscr_3(F) \otimes F^\times_\Qbb \arrow[d,"\partial^{4,2}"] \arrow[r,"\alpha_4(2)"] & \Bscr_2(F) \otimes \bigwedge^2 F^\times_\Qbb \arrow[d,"\partial^{4,3}"] \arrow[r, "\alpha_4(3)"] & \bigwedge^4 F^\times_\Qbb \arrow[d, "\partial^{4,4}"] \arrow[ldd, "H" near end]\\
			& \bigoplus_{p\in X^1} \Bscr_3(k(p)) \arrow[r,"\alpha_3(1)"] & \bigoplus_{p \in X^1} \Bscr_2(k(p)) \otimes k(p)^\times_\Qbb \arrow[d,"\partial^{3,2}"] \arrow[r,"\alpha_3(2)"] &\bigoplus_{p \in X^{1}} \bigwedge^3 k(p)^\times_\Qbb \arrow[d,"\partial^{3,3}"] \\
			& & \bigoplus_{q \in X^2} \left(\Bscr_2(\overline{k(q)})\right)^{G_q} \arrow[r,"\alpha_2(1)"']& \bigoplus_{q \in X^2} \bigwedge^2 \left(\overline{k(q)}^\times_\Qbb\right)^{G_q},
		\end{tikzcd}
	\end{equation}
	where $\Bscr_4(F)$ is placed in horizontal degree 1 and vertical degree 0, the maps $\alpha_i(j)$ are defined in \eqref{weightnGoncomplex}, and the residue maps $\partial^{4,i} := \oplus_{p\in X^1} \partial_p^{4, i}$ 
    are given by
	\begin{equation}\label{res4}
		\begin{aligned}
			&\partial_p^{4, 2} : \{f\}_3 \otimes g \mapsto \mathrm{ord}_p(g) \{f(p)\}_3, \\
			&\partial_p^{4, 3} : \{f\}_2 \otimes g \wedge h 
			\mapsto \{f(p)\}_2 \otimes T_p\{g, h\},\\
			&\partial_p^{4, 4} : f \wedge g \wedge h \wedge \ell \mapsto \mathrm{ord}_p(f) g_p \wedge h_p \wedge \ell_p + \mathrm{ord}_p(g) f_p \wedge h_p \wedge \ell_p  \\
			&\hspace{5cm} + \mathrm{ord}_p(h) f_p \wedge g_p\wedge\ell_p + \ord_p(\ell) f_p \wedge g_p \wedge \ell_p,
		\end{aligned}
	\end{equation}
	where $T_p$ is the tame symbol at $p \in X^1$, i.e., 
	\begin{equation*}
		T_p\{f, g\} := (-1)^{\ord_p(f) \ord_p(g)}\left(\dfrac{f^{\ord_p(g)}}{g^{\ord_p(f)}}\right)(p),
	\end{equation*}
	and $f_p \in k(p)^\times$ is the image of $f \pi_p^{-\ord_p(f)}$ under the canonical map $F^\times  \to k(p)^\times$. One can show that these residues maps do not depend on the choice of uniformizers (see \cite[Section 2.3]{Gon98}). Let us write down the total complex associated to \eqref{bicomplex4}, in degrees 1 to 4
	\begin{equation*}
		\begin{aligned}
			\Bscr_4(F) \xrightarrow{\alpha_4(1)_X}  \Bscr_3(F) \otimes F_\Qbb^\times \xrightarrow{\alpha_4(2)_X} (\Bscr_2(F) \otimes \bigwedge^2 &F_\Qbb^\times) \oplus \bigoplus_{p \in X^1} \Bscr_3(k(p)) \\
			&\xrightarrow{\alpha_4(3)_X} \bigwedge^4 F_\Qbb^\times \oplus (\bigoplus_{p\in X^1} \Bscr_2(k(p)) \otimes k(p)_\Qbb^\times),
		\end{aligned}
	\end{equation*}
	where 
	\begin{equation}\label{map}
		\begin{aligned}
			\alpha_4(1)_X &:= \alpha_4(1), \\
			\alpha_4(2)_X (u) &:= (\alpha_4(2)(u), \partial^{4,2}(u)),\\
			\alpha_4(3)_X (u_1, u_2) &:= (\alpha_4(3)(u_1), \partial^{4, 3}(u_1) - \alpha_3(1)(u_2)),
		\end{aligned}
	\end{equation} for $u \in \Bscr_3(F) \otimes F_\Qbb^\times$ and $u_1 \in \Bscr_2(F) \otimes \bigwedge^2 F_\Qbb^\times$, $u_2 \in  \oplus_{p\in X^1}\Bscr_3(k(p))$. We have $H^1(\Gamma(X, 4)) \simeq H^1(\Gamma(F, 4))$. There are canonical maps
    \begin{equation}
        H^2(\Gamma(X, 4)) \to H^2(\Gamma(F, 4)), \quad [u] \mapsto [u], 
    \end{equation}
    and 
    \begin{equation}
        H^3(\Gamma(X, 4)) \to H^3(\Gamma(F, 4)), \quad [(u_1, u_2)] \mapsto [u_1].
    \end{equation}
    One can check that these maps are well-defined. In fact, if $(u_1, u_2)$ and $(v_1, v_2)$ define the same class in $H^3(\Gamma(X, 4))$, then $u_1 - v_1 \in \Im \alpha_4(2)$, so  that $u_1$ and $v_1$ define the same class in  $H^3(\Gamma(F, 4))$.  Moreover, if $U \hookrightarrow X$ is an open regular subscheme of $X$, then there is a canonical map 
	$$H^3(\Gamma(X, 4)) \to H^3(\Gamma(U, 4)), \quad [(u_1, u_2)] \mapsto [(u_1, \pr u_2)],$$
	where $\pr : \oplus_{x\in X^1} \Bscr_3(k(p)) \to \oplus_{x\in U^1} \Bscr_3(k(p))$ is the canonical projection. One can show that this map is well-defined. We have the following exact sequence 
	\begin{equation}\label{10.4.7}
		\begin{aligned}
			0 \to H^2(\Gamma(X, 4)) \to H^2(\Gamma(F, 4)) &\xrightarrow{\overline{\partial^{4,2}}} \oplus_{p\in X^1} H^1(\Gamma(k(p), 3)) \\
			&\xrightarrow{\delta} H^3(\Gamma(X, 4)) \to H^3(\Gamma(F, 4)) \xrightarrow{\overline{\partial^{4,3}}} \oplus_{p\in X^1} H^2(\Gamma(k(p), 3)),
		\end{aligned}
	\end{equation}
	where the map $\delta := \oplus_{p\in X^1}\delta_p$ with 
	$$\delta_p : H^1(\Gamma(k(p), 3)) \to H^3(\Gamma(X, 4)), \quad [u] \mapsto [(0, u)].$$
	This map is well-defined. Indeed, if $[u] \in H^1(\Gamma(k(p), 3))$, then $\alpha_3(1) (u) = 0$, so that $(0, u) \in \Ker \alpha_4(3)_X$. Therefore, $(0,u)$ defines an element in $H^3(\Gamma(X, 4))$. And if $[u_1] = [u_2]$  in $H^1(\Gamma(k(p), 3))$, we have $u_1 = u_2$ in $\Bscr_3(k(p))$, hence $(0, u_1)$ and $(0, u_2)$ define the same class in $H^3(\Gamma(X, 4))$.
	\begin{rmk}\label{conditionLambda}
		Let $\lambda = \sum_j c_j \{f_j\}_2 \otimes g_j \wedge h_j  \in \Bscr_2(F) \otimes \bigwedge^2 F_\Qbb^\times$  be a 3-cocycle in $\Gamma(F, 4)$. Let $S$ be the closed subscheme of $X$ consisting of all zeros and poles of $g_j, h_j$ for all $j$. Note that $\partial^{4, 3}_p(\lambda) = 0$ for all $p \notin S$. So if $\partial^{4,3}_p(\lambda) = 0$ for all $p \in S$, the element
		\begin{equation}\label{Lambda}
			\Lambda := (\lambda, 0) \in (\Bscr_2(F) \otimes \bigwedge^2 F_\Qbb^\times) \oplus \bigoplus_{p \in X^1} \Bscr_3(k(p))
		\end{equation}
		defines a class in $H^3(\Gamma(X, 4))$.  Goncharov (\cite{GR22}) conjectured that there is a functorial isomorphism 
		\begin{equation}\label{15.1.3}
			H^3(\Gamma(X, 4)) \xrightarrow{\simeq} H^3_\Mcal(X, \Qbb(4)).
		\end{equation}
		Therefore, $\Lambda$ should give rise to an element in the motivic cohomology $H^3_\Mcal(X, \Qbb(4))$.
	\end{rmk}
	
	\begin{rmk}
	    Note that in the three-variable case (see \cite[Section 3.4]{Tri23}), we constructed a map from the cohomology of Goncharov's polylogarithmic complex to the motivic cohomology of smooth projective curves
        $$\beta : H^2(\Gamma(C, 3)) \to H^2_\Mcal(C, \Qbb(3)),$$
        which is compatible with taking regulators, i.e., the following diagram commutes
        \begin{equation*}
        \xymatrix{H^2(\Gamma(C, 3)) \ar[r]^{\beta} \ar[d]_{r_3(2)_C} & H^2_\Mcal(C, \Qbb(3)), \ar[dl]^{\frac{1}{2}\reg^{2, 3}_C}\\
            H^1(C(\Cbb), \Rbb(2))^+&}
        \end{equation*}
        where $r_3(2)_C$ is the Goncharov regulator map (see \cite{Gon98}) and $\reg^{2,3}_C$ is the Beilinson regulator map.
	\end{rmk}

	We have the following result on the Mahler measure of exact polynomials in four variables.
	
	\begin{thm}\label{4varsmain}
		Let $P(x, y, z, t)\in \Qbb[x, y, z, t]$ be a nonzero irreducible polynomial. We assume that \begin{equation}\label{15.1.9}
		x\wedge y \wedge z \wedge t = \sum_j  c_j f_j \wedge (1-f_j) \wedge g_j \wedge h_j \quad \text{in } \bigwedge^4 \Qbb(V_P)^\times_\Qbb,
	\end{equation}
	for some  $f_j \in \Qbb(V_P)^\times \setminus \{1\}$ and $g_j, h_j \in \Qbb(V_P)^\times$. Let $X$ be a smooth compactification of $W_P$. Denote by $Y$ the regular open subscheme of $W_P$ defined in \eqref{Y}. Assume that $\Gamma \subset V_P^\reg$ and that $\partial \Gamma \subset Y(\Cbb)$. Let $\lambda$ be the 3-cocycle of $\Gamma(F, 4)$ defined in Definition \ref{310}. Suppose that all the residues $\partial^{4, 3}_p(\lambda) = 0$ for all $p \in S$, where $S$ is the closed subscheme of $X$ consisting of zeros and poles of $g_j, h_j, g_j \circ \tau, h_j \circ \tau$ for all $j$.  Under Goncharov's conjecture \ref{15.1.3}, the element 
    \begin{equation}\label{Lambdamain}
    \Lambda := (\lambda, 0) \in (\Bscr_2(F) \otimes \bigwedge^2 F_\Qbb^\times) \oplus \bigoplus_{p\in X^1} \Bscr_3(k(p))
    \end{equation}
		 defines a class in  the motivic cohomology $H^3_\Mcal(X, \Qbb(4))$ and 
		\begin{equation}\label{mainid}
			\m(P) - \m(\tilde P) = -\dfrac{1}{(2\pi i)^3} \left<[\partial \Gamma], \reg^{3, 4}_X  (\Lambda)\right>_X,
		\end{equation}
		where $\tilde P \in \Qbb[x, y, z]$ is the leading coefficient of $P$ seen as a polynomial in $t$, and the pairing is given by
		$$\left<\cdot, \cdot \right>_X : H_2(X(\Cbb), \Qbb)^- \times H^2(X(\Cbb), \Rbb(3))^+ \to \Rbb(3).$$
	\end{thm}
    
	\begin{proof}
    Under the conditions that $\Gamma \subset V_P^\reg$ and $\partial \Gamma \subset Y(\Cbb)$, the boundary $\partial \Gamma$ defines an element in $H_2(Y(\Cbb), \Qbb)^-$ (see Theorem \ref{7.2.8}).  Let $\iota : Y \hookrightarrow X$ be the canonical embedding. We have $[\partial \Gamma] = \iota_*[\partial \Gamma]$, which therefore defines an element in $H_2(X(\Cbb), \Qbb)^-$.
    
    Since all the residues $\partial_p^{4,3}(\lambda)$ vanish, the element $\Lambda$ defines a class in $H^3(\Gamma(X, 4))$ (see Remark \ref{conditionLambda}). Under Goncharov's conjecture \ref{15.1.3}, $\Lambda$ defines an element in the motivic cohomology $H^3_\Mcal(X, \Qbb(4))$  and we have the following commutative diagram
		\begin{equation}\label{232}
			\xymatrix{
				H^3(\Gamma(X, 4))\ar[r]^{\text{conj}}_\simeq \ar[d]_{\iota^*} & H^3_\Mcal(X, \Qbb(4))\ar[r]^{\reg^{3,4}_X}\ar[d]^{\iota^*} & H^2(X(\Cbb), \Rbb(3))^+ \ar[d]^{\iota^*}\\
				H^3(\Gamma(Y, 4)) \ar[r]^{\text{conj}}_\simeq  & H^3_\Mcal(Y, \Qbb(4))\ar[r]^{\reg^{3, 4}_Y} & H^2(Y(\Cbb), \Rbb(3))^+.
			}
		\end{equation}
		We then obtain the following identities 
		\begin{equation*}
			\left<[\partial \Gamma], [\rho(\lambda)]\right>_Y = \left<[\partial \Gamma], \reg^{3, 4}_Y(\iota^*\Lambda) \right>_Y = \left<[\partial \Gamma], \iota ^*\reg^{3, 4}_X(\Lambda) \right>_X = \left<\iota_*[\partial \Gamma], \reg^{3, 4}_X(\Lambda) \right>_X = \left<[\partial \Gamma], \reg^{3, 4}_X (\Lambda)\right>_X,
		\end{equation*}
		The theorem then follows from Theorem \eqref{7.2.8}.
	\end{proof}

	\subsection{Beilinson's conjectures for $K3$ surfaces}\label{section K3} By an algebraic surface we mean a separated, finite type, geometrically connected scheme of dimension 2 over a field. Let $k$ be an arbitrary field of characteristic 0. Recall that a non-singular projective surface  $X$ over $k$ is called an algebraic $K3$ surface if the canonical sheaf $\omega_{X/k}$ is isomorphic to $\Ocal_X$ and $H^1(X, \Ocal_X) = 0$. By Serre's duality and the Hirzebruch-Riemann-Roch theorem (see \cite[Appendix A, Theorem 4.1]{Har77}), one can show that the Hodge diamond of an algebraic $K3$ surface $X$ is 
	\begin{center}
		\begin{tikzcd}[column sep=tiny,row sep=tiny]
			&&h^{2, 2}&&\\
			&h^{2,1}&&h^{1, 2}&\\
			h^{2, 0}&&h^{1, 1}&&h^{0, 2}\\
			&h^{1,0}&&h^{0, 1}&\\
			&&h^{0, 0}&&
		\end{tikzcd} $\hspace{2cm}$ \begin{tikzcd}[column sep=tiny,row sep=tiny]
			&&1&&\\
			&0&&0&\\
			1&&20&&1,\\
			&0&&0&\\
			&&1&&
		\end{tikzcd}
	\end{center}
	where $h^{p, q} := \dim H^q(X, \Omega_{X/k}^p)$. Recall that the Néron-Severi group of an algebraic surface $X$ is 
	$$\NS(X) = \Pic(X)/\Pic^0(X),$$
	where $\Pic^0(X)$ is the subgroup of $\Pic(X)$ consisting of the line bundles which are algebraically equivalent to 0. If $X$ is an algebraic $K3$ surface, the natural surjection is an isomorphism (see \cite[Proposition 2.4]{Huy15})
	\begin{equation}\label{PicNS}
		\Pic(X) \xrightarrow{\simeq}\NS(X).
	\end{equation}
	Recall that if $k$ is a field of characteristic 0, the Picard number of $X$ is defined as
	\begin{equation}\label{259}
		\rho (X_{\bar k}) := \rk \NS(X_{\bar k}) \le 20.
	\end{equation}
	When $\rho(X_{\bar k}) = 20$, we call $X$ a \textit{singular $K3$ surface}.

	Let $X$ be a smooth projective surface over  $k$. Recall that there is a direct decomposition of the Chow motive $h(X)$ in the category $\CHM(k, \Qbb)$ 
	\begin{equation}\label{CKdecom}
		h(X) \simeq \bigoplus_{i=0}^4 h^i(X),
	\end{equation}
	where $h_i(X) = (X, p_i(X), 0)$ and $p_i(X) \in A^2(X \times X)$ are certain projectors (see \cite{Scho94} for more details). The projector $p_2$ is uniquely decomposed into the transcendental and algebraic parts (see  \cite[Proposition 14.2.3]{KMP13}\footnote{Notice that \cite{KMP13} uses the covariant Chow motives category $\Mcal_\rat$, whose objects are the same as those in the  category of Chow motives $\CHM(k, \Qbb)$, and the morphisms are defined by $\Hom_{\Mcal_\rat}((X, p, m), (Y, q, n)) = q \Corr^{n-m}(Y, X) p.$}). 
	$$p_2 = p_{2, \alg} + p_{2, \tr}.$$
We then have the following decomposition in $\CHM(k, \Qbb)$
	\begin{equation}\label{218}
		h^2(X) = h^2_{\alg}(X) \oplus h^2_{\tr}(X),
	\end{equation}
	where $h^2_{\alg} (X) = (X, p_{2, \alg}, 0)$ and $h^2_{\tr}(X) = (X, p_{2, \tr}, 0).$ Consequently, if $H$ is a Weil cohomology (e.g, singular cohomology, de Rham cohomology, étale cohomology,...), we then have the following decomposition 
	\begin{equation}\label{cohotrans}
		H^2(X) = H^2(X)_{\alg} \oplus H^2(X)_{\tr}.
	\end{equation}
	One calls $H^2(X)_{\tr}$ the \textit{transcendental cohomology} of $X$. Let $E$ be a finite Galois extension over $k$ such that the action of $\Gal(\bar k/k)$ factors through $\Gal(E/k)$ on the Néron-Severi group $\NS(X_{\bar k})$. For $ 1 \le j \le \rho(X_{\bar k})$, there are projectors $p_{2, \alg, j} \in A^2((X\times X)_E)$ such that 
    \begin{equation}\label{halg}
		h^2_{\alg} (X_E) = \bigoplus_{j=1}^{\rho(X_{\bar k})} h^2_{\alg, j}(X_E),
	\end{equation}
    where $h^2_{\alg, j}(X_E) := (X_E, p_{2, \alg, j}, 0) \simeq \Lbb_E = (\Spec E, \id, -1) \in \CHM(E, \Qbb)$ for $1 \le j \le \rho(X_E)$.

	Denote by $M^i(X), M^2_{\alg} (X), M^2_{\tr} (X)$ the images of $h^i(X), h^2_{\alg}(X), h^2_{\tr} (X)$ respectively under the fully faithful contravariant functor from the category of Chow motives  $\CHM(k, \Qbb)$ to the category of Voevodsky geometrical motives $\DM_\gm(k, \Qbb)$ (see \cite[Proposition 2.1.4, Proposition 4.2.6]{VSF00}). We have the corresponding direct sum decompositions in the category $\DM_{\mathrm{gm}}(k, \Qbb)$
	\begin{equation}\label{CKdecom2}
		M(X) = \bigoplus_{i=0}^4 M^i(X), \quad  M^2(X) = M^2_{\alg} (X) \oplus M^2_{\tr} (X).
	\end{equation}

	\begin{definition}[\textbf{Transcendental  and algebraic parts of  motivic cohomology}] Let $X$ be a smooth projective surface over $k$. The algebraic and transcendental parts of the motivic cohomology of $X$ are respectively defined to be
		\begin{equation*}
			H^i_\Mcal(X, \Qbb(j))_{\alg} := \Hom_{\DM_\mathrm{gm}(k, \Qbb)}(M^2_{\alg}(X), \Qbb(j)[i]), \quad 
			H^i_\Mcal(X, \Qbb(j))_{\tr} := \Hom_{\DM_\mathrm{gm}(k, \Qbb)}(M^2_{\tr}(X), \Qbb(j)[i]).
		\end{equation*}

	\end{definition}

	\begin{prop}\label{370}
		Let $X$ be an algebraic $K3$ surface over a number field $k$. The motivic cohomology $H^3_\Mcal(X, \Qbb(4))$ depends only on the motive $M^2(X)$, i.e.,
		\begin{equation}
			H^3_\Mcal(X, \Qbb(4)) = \Hom_{\DM_\gm(k, \Qbb)} (M^2(X), \Qbb(4)[3]).
		\end{equation}
		In particular, we have 
		\begin{equation}
			H^3_\Mcal(X, \Qbb(4)) = H^3_\Mcal(X, \Qbb(4))_{\tr} \oplus H^3_\Mcal(X, \Qbb(4))_{\alg}.
		\end{equation}
	\end{prop}

	\begin{proof}
		We have
		\begin{equation*}
			\begin{aligned}
				H^3_\Mcal(X, \Qbb(4)) \stackrel{\mathrm{def}}{=} \Hom_{\DM_\gm(k, \Qbb)} (M(X), \Qbb(4)[3]) &= \bigoplus_{i=0}^4 \Hom_{\DM_\gm(k, \Qbb)} (M^i(X), \Qbb(4)[3]).
			\end{aligned}
		\end{equation*}
		Denote by $r_1, r_2$ the number of real and complex embeddings $k \hookrightarrow \Cbb$, respectively. 
We have 		\begin{equation*}		\begin{aligned}
	\Hom_{\DM_\gm(k, \Qbb)} (M^0(X), \Qbb(4)[3]) =    \Hom_{\DM_\gm(k, \Qbb)} (\Spec k, \Qbb(4)[3]) = H^3_\Mcal(\Spec k, \Qbb(4)) = 0,
			\end{aligned}
		\end{equation*}
        where the last equality follows from Borel's theorem (\cite{Bor74}).		By symmetry, one also has $$ \Hom_{\DM_\gm(k, \Qbb)} (M^4(X), \Qbb(4)[3])  = 0.$$  Since 
        $\End_{\CHM(k, \Qbb)}(h^1(X)) \simeq \End(\Pic_X^0) \otimes \Qbb$
        (see e.g., \cite{Scho94}), and the fact that the Picard variety 
		$\Pic_X^0$ is trivial, we get $h^1(X) = 0$. By symmetry, we also have $h^3(X) = 0$. Finally, we have 
		\begin{equation*}
			\begin{aligned}
				H^3_\Mcal(X, \Qbb(4)) &= \Hom_{\DM_\gm(k, \Qbb)} (M^2(X), \Qbb(4)[3]) \\
				&\simeq \Hom_{\DM_\gm(k, \Qbb)} (M^2_{\tr}(X), \Qbb(4)[3]) \oplus \Hom_{\DM_\gm(k, \Qbb)} (M^2_{\alg}(X), \Qbb(4)[3]) \\
				&= H^3_\Mcal(X, \Qbb(4))_{\tr} \oplus H^3_\Mcal(X, \Qbb(4))_{\alg}.
			\end{aligned}
		\end{equation*}
	\end{proof}

	\begin{rmk}[{\textbf{Transcendental and algebraic parts of Deligne-Beilinson cohomology}}]
		Let $X$ be a smooth projective surface over $\Rbb$. By \eqref{2.1.8}, we can define the algebraic and transcendental parts  of the Deligne-Beilinson cohomology of $X$ as follows
		\begin{equation*}
			H^3_\Dcal(X, \Rbb(4))_{\alg} := H^2(X(\Cbb), \Rbb(3))^+ \cap H^2(X(\Cbb), \Rbb(3))_{\alg}, 
		\end{equation*}
		and 
		\begin{equation*}
			H^3_\Dcal(X, \Rbb(4))_{\tr} := H^2(X(\Cbb), \Rbb(3))^+ \cap H^2(X(\Cbb), \Rbb(3))_{\tr},
		\end{equation*}
        where $``+"$ means the invariant part under the action of de Rham involution $F_\dR$ (see \eqref{82}). One can see \cite[Definition 12.2.5]{Tri24} for the general definition. 
	\end{rmk}

	We have the following result for the  transcendental and algebraic parts of the Deligne-Beilinson cohomology of singular $K3$ surfaces.
	\begin{prop}\label{DBK3}
		Let $X$ be a singular $K3$ surface over $\Qbb$. Then we have 
		$$\dim_\Rbb H^3_\Dcal(X_\Rbb, \Rbb(4))_{\tr} = 1.$$
		Furthermore, if $X$ has Picard rank 20 over $\Qbb$ (i.e., the Néron-Severi group is  generated by 20 divisors defined over $\Qbb$), then 
		$$\quad \dim_\Rbb H^3_\Dcal(X_\Rbb, \Rbb(4))_{\alg} = 20.$$ 
	\end{prop}
	\begin{proof}
		Recall that we have the following direct decomposition (see \eqref{cohotrans})
		\begin{equation*}
			H^2(X(\Cbb), \Rbb(3)) = H^2(X(\Cbb), \Rbb(3))_{\alg} \oplus H^2(X(\Cbb), \Rbb(3))_{\tr},
		\end{equation*}
		with
		\begin{equation}\label{239}
			H^2(X(\Cbb), \Rbb(3))_{\alg} \simeq  NS(X(\Cbb))\otimes \Rbb(3) \simeq \Rbb(3)^{\rho(X)} = \Rbb(3)^{20}.
		\end{equation}
		Therefore, 
		$$\dim H^2(X(\Cbb), \Rbb(3))_{\tr} = 22-20 = 2.$$
		On the other hand, we have
		\begin{equation*}
			H^2(X(\Cbb), \Rbb(3))_{\tr} = \{\omega \in H^2(X(\Cbb), \Cbb)_{\tr} \ |\ \bar \omega = -\omega\}.
		\end{equation*}
		Indeed, if $\omega \in H^2(X(\Cbb), \Rbb(3))_{\tr}$, we have $\omega = (2\pi i)^3 \omega'$ with some $\omega' \in H^2(X(\Cbb), \Rbb)_{\tr}$, then $$\bar \omega = - (2\pi i) \omega' = - \omega.$$ As $X$ is a singular $K3$ surface over $\Qbb$, we have 
		$$H^2(X, \Cbb)_{\tr} = \Omega^2(X(\Cbb)) \oplus  \overline{\Omega^2(X(\Cbb))},$$
		with $\dim_\Cbb \Omega^2(X(\Cbb)) = \dim_\Cbb \overline{\Omega^2(X(\Cbb))} = 1$.  Let $\omega$ be a generator of $\Omega^2(X(\Cbb))$ that is defined over $\Qbb$. We set 
		$$\omega_1 :=  \omega - \bar \omega \text{ and } \omega_2 := i(\omega + \bar \omega).$$ It is clear to see $\omega_1, \omega_2 \in H^2(X(\Cbb), \Rbb(3))_{\tr}$. As $\omega$ is defined over $\Qbb$, we have 
		\begin{equation*}
			\begin{aligned}
				F^{\dR}(\omega_1) &\stackrel{\mathrm{def}}{=} c^*(\bar \omega_1) = c^*(\bar \omega) - c^*(\omega) = \omega - \bar \omega = \omega_1,\\
				F^{\dR}(\omega_2) &\stackrel{\mathrm{def}}{=} c^*(\bar \omega_2) = -i (c^*(\bar \omega) + c^*(\omega) = -i(\omega + \bar \omega) = -\omega_2.\\
			\end{aligned}
		\end{equation*}
		This implies that 
		\begin{equation*}
			\begin{aligned}
				\dim_\Rbb H^3_\Dcal(X_\Rbb, \Rbb(4))_{\tr} = \dim_\Rbb H^2 (X(\Cbb), \Rbb(3))^+_{\tr} = 1.
			\end{aligned}
		\end{equation*}
		
		If $X$ has Picard rank $20$ over $\Qbb$, the Néron-Severi group $\NS(X(\Cbb))$ has a basis $\{[D_i], 1\le i \le 20\}$, where $D_i$ are irreducible smooth projective curves defined over $\Qbb$. The complex conjugation then acts on the set of complex points $D_i(\Cbb)$ and reverses the orientation of $D_i(\Cbb)$ for all $1 \le i \le 20$. On the other hand, we have $\dim H^2(X(\Cbb), \Rbb(3))_{\alg} = 20$. This implies that 
		$$\dim_\Rbb H^3_{\Dcal}(X_\Rbb, \Rbb(4))_{\alg} = \dim_{\Rbb} H^2(X(\Cbb), \Rbb(3))_{\alg}^+ = \dim_\Rbb H^2(X(\Cbb) , \Rbb)_{\alg}^- = 20.$$
	\end{proof}

	Before stating Beilinson’s conjecture, we briefly recall the definition of the 
$L$-function associated with pure motives Let $X$ be a smooth projective variety over $\Qbb$ of dimension $n$.

\begin{definition}[\textbf{$L$-factors at primes}, {\cite[\textsection \ 1.4]{Nek13}}]
Let $p$ be a prime number. For $0\le i \le 2n$, we set
\begin{equation*}
    L_p(h^i(X), s) = \det (1 - \mathrm{Frob}_p p^{-s} | H^i_\ell(X)^{I_p})^{-1},
\end{equation*}
where $\ell \neq p$ is a prime number,  $\mathrm{Frob_{p}}\in \Gal(\bar \Qbb/\Qbb)$ is a Frobenius element at $p$ acting on the étale realization 
\begin{equation*}
   H^i_\ell(X) :=  H^i_{\et}(X_{\bar \Qbb}, \Qbb_\ell),
\end{equation*}
and $I_p$ is the inertia group at $p$. If $X$ is a smooth projective surface, we define
\begin{equation*}
    L_p(h^2_{\tr}(X), s) = \det (1 - \mathrm{Frob}_p p^{-s} | H^2_\ell(X)_{\tr}^{I_p})^{-1},
\end{equation*}
where $H^2_\ell(X)_{\tr}$ denotes the transcendental part $H^2_\et(X_{\bar \Qbb}, \Qbb_\ell)_{\tr}$ (see \eqref{218}).
\end{definition}

\begin{rmk}
If $X$ has good reduction at $p$, i.e., if $X$ admits a projective model over $\Zbb$ whose reduction modulo $p$ is smooth over $\Fbb_p$,  then $I_p$ acts trivially on $H^i_\ell(X)$ and it was shown by Deligne that $L_p(h^i(X), s)$ have integer coefficients and is independent of $\ell \neq p$ (\cite[\textsection \ 1.4]{Nek13}, \cite[Section 5.6.2]{Kah20}). If $X$ has bad reduction at $p$ (there are finitely many such primes $p$), Serre conjectured that $L_p(h^i(X),s)$ is independent of the choice of $\ell$ and has integer coefficients (see \cite[Conjecture 5.45]{Kah20}). This conjecture holds if $i \in \{0, 1, 2n-1, 2n\}$ (see
    \cite[Theorem 5.46]{Kah20}).
\end{rmk}

\begin{definition}[{\textbf{$L$-function}, \cite[\textsection \ 1.5]{Nek13}}]\label{301}
The $L$-function associated to the motive $h^i(X)$ is defined (conjecturally) by
\begin{equation*}
    L(h^i(X), s) = \prod_{p \ \mathrm{prime}} L_p(h^i(X), s).
\end{equation*}
If $X$ is a smooth projective surface, we define (conjecturally)
\begin{equation*}
     L(h^2(X)_{\tr}, s) = \prod_{p \ \mathrm{prime}} L_p(h^2_{\tr}(X), s).
\end{equation*}
\end{definition}

	Notice that the real Deligne-Beilinson cohomology of smooth varieties over $\Rbb$ or $\Cbb$ can be defined as an $\mathbb{E}$-cohomology, where $\mathbb{E}$ is a certain motivic ring spectrum in the $\Abb^1$-derived category (see e.g., \cite[Appendix A.2, A.3]{BZ20}). Therefore, the Beilinson regulator map sends the transcendental (resp. algebraic) part of motivic cohomology to the transcendental (resp. algebraic) part of Deligne-Beilinson cohomology.
    Moreover, as the Deligne–Beilinson cohomology of singular $K3$ surfaces is one-dimensional (see Proposition \ref{DBK3}), Beilinson’s conjecture for the transcendental part of their motivic cohomology can be stated as follows. 
	
	\begin{conj}\label{2.7}
		Let $X$ be a singular $K3$ surface over $\Qbb$. Let $\xi \in H^{3}_\Mcal(X, \Qbb(4))_{\tr}$ be a nontrivial element and  $\gamma$ be a generator of $H_2(X(\Cbb), \Qbb)_{\tr}^-$. We have
		\begin{equation} 
			\dfrac{1}{(2 \pi i)^3} \left<\gamma, \reg^{3, 4}_{X}(\xi)\right>_{X, \tr} = a \cdot L'(h^2_{\tr}(X), -1), \quad a \in \Qbb^\times,
		\end{equation}
		where  the pairing is given by 
		\begin{equation}\label{241}
			\left< \cdot, \cdot \right>_{X, \tr} : H_2(X(\Cbb), \Qbb)_{\tr}^- \times H^2(X(\Cbb), \Rbb(3))_{\tr}^+ \to \Rbb(3), 
		\end{equation}
        and $h^2_{\tr}(X)$ is the transcendental part of the Chow motive $h^2(X)$ (see \eqref{218}).
	\end{conj}

	Let us recall the following result of Livné.
	\begin{thm}[{\textbf{Livné's modularity theorem}, \cite{Liv94}}] Let $X$ be a singular $K3$ surface over $\Qbb$ (not necessarily having Picard rank 20 over $\Qbb$). Then there is a modular form of weight 3 such that
		$$L(h^2_{\tr}(X), s) = L(f, s).$$
	\end{thm}

	Let $X$ be a smooth projective surface over $\Qbb$. The Néron-Severi group $\NS(X(\Cbb))$ is endowed with an intersection pairing, which is compatible with the cup product of the singular cohomology $H^2(X(\Cbb), \Zbb)$. Therefore, $\NS(X(\Cbb))$ is  a sublattice of $H^2(X(\Cbb), \Zbb)$.  Since this intersection pairing is non-degenerate, one defines $\Tbb(X(\Cbb)) = \NS(X(\Cbb))^\bot$,  called the \textit{transcendental lattice}. Note that $\Tbb(X(\Cbb)) \otimes \Qbb \simeq H^2(X(\Cbb), \Qbb)_{\tr}$ (see \eqref{cohotrans}). Schütt gave explicitly the modular form in Livné's theorem in terms of the determinant of the transcendental lattice when $X$ has Picard rank 20 over $\Qbb$.

	\begin{thm}[{\cite[Theorem 4, Lemma 17]{Schu08}}]\label{liv}
		Let $X/\Qbb$ be a singular $K3$ surface with Picard rank 20 over $\Qbb$. Denote by $d = |\det \Tbb(X(\Cbb))|$. Let $K = \Qbb(\sqrt{-d})$ and $d_K$ be the discriminant of $K$. Suppose that $d_K \neq 3, 4$, then there exists a newform $f$ of weight 3 and level  \begin{equation*}
			D = \begin{cases}
				-d_K & \text{if } 4\nmid d_K,\\
				-\dfrac{d_K}{4} &\text{if } 4\mid d_K,
			\end{cases}
		\end{equation*}
		such that 
		\begin{equation*}
			L(h^2_{\tr}(X), s) = L(f, s).
		\end{equation*}
	\end{thm}
	
	We then have the following conjecture which is similar to the case of smooth projective curves of genus 1 (see \cite[Conjecture 1.11]{Tri23}).
	
	\begin{conj}[\textbf{Beilinson's conjecture for the transcendental part of motivic cohomology}]\label{242}
		Let $X$ be a singular $K3$ surface with Picard rank $20$ over $\Qbb$.  Let $\xi \in H^{3}_\Mcal(X, \Qbb(4))_{\tr}$ be a nontrivial element and $\gamma$ be a generator of $H_2(X(\Cbb), \Qbb)^-_{\tr}$. We have
		\begin{equation} 
			\dfrac{1}{(2 \pi i)^3} \left<\gamma, \reg^{3, 4}_{X}(\xi)\right>_{X, \tr} = a \cdot L'(f, -1), \quad a \in \Qbb,
		\end{equation}
		where  $f$ is the modular form defined in Theorem \ref{liv}.
	\end{conj}
    
The following proposition shows that Beilinson's conjecture for the algebraic part of the motivic cohomology of singular $K3$ surfaces reduces to Borel's theorem. 
\begin{prop}\label{alg}
    Let $X/\Qbb$ be a singular $K3$ surface with Picard rank 20 over $\Qbb$. Let $\xi \in H^3_\Mcal(X, \Qbb(4))_{\alg}$ be a nontrivial element and $\gamma \in H_2(X(\Cbb), \Qbb)^-_{\alg}$ be a generator, we have  
	\begin{equation}\label{252}
		\dfrac{1}{(2\pi i)^3} \left<\gamma, \reg^{3, 4}_X(\xi)\right>_{X, \alg} = a \cdot \zeta'(-2), \quad a \in \Qbb.
	\end{equation}
\end{prop}
\begin{proof}
 We fix an orthogonal basis
	$$\{[D_j], 1 \le j \le 20\}$$ of $\NS(X)$,  where $[D_j]$ denotes the algebraic equivalence class corresponding to the divisor $D_j$ defined over $\Qbb$. Then we  have the following direct sum decomposition in the category $\CHM(\Qbb)$ of Chow motive over $\Qbb$ with $\Qbb$-coefficients (see \eqref{halg})
	\begin{equation}\label{246}
		h^2_{\alg}(X) \simeq \bigoplus_{j=1}^{20} h^2_{\alg, j}(X), 
	\end{equation}
	where $h^2_{\alg, j}(X) = (X, p_{2, \alg, j}(X), 0) \simeq \Lbb = (\Spec \Qbb, \id, -1)$, and $p_{2, \alg, j}(X) \in \Corr^0(X, X)$ are projectors corresponding to $D_j$ for $1 \le j\le 20$. Recall that the fully faithful functor $\CHM(\Qbb) \to \DM_{\gm}(\Qbb)$ sends $\Lbb$ to $\Qbb(1)[2]$ (see \cite[Chapter 5, Proposition 2.1.4]{VSF00}). We then have the following direct sum decomposition in $\DM_{\gm}(\Qbb)$
	\begin{equation}
		M_{\alg}^2 (X) \simeq  \bigoplus_{j=1}^{20} M^2_{\alg, i},
	\end{equation}
	where $M^2_{\alg, j}\simeq \Qbb(1)[2]$ is in the category $\DM_\gm(\Qbb)$. Thus 
	\begin{equation}\label{248}
		H^3_\Mcal(X, \Qbb(4))_{\alg} \simeq \bigoplus_{j=1}^{20} H^3_\Mcal(X, \Qbb(4))_{\alg, j}, 
	\end{equation}
	where 
	\begin{equation}\label{250}
		\begin{aligned}
			H^3_\Mcal(X, \Qbb(4))_{\alg, j} &= \Hom_{\DM_{\gm}(\Qbb)}(M^2_{\alg, j}(X), \Qbb(4)[3])\\
			&\simeq \Hom_{\DM_{\gm}(\Qbb)} (\Qbb(1)[2], \Qbb(4)[3])\\
			&\simeq \Hom_{\DM_{\gm}(\Qbb)} (\Qbb(0), \Qbb(3)[1])\\
			&= H^1(\Spec \Qbb, \Qbb(3)).
		\end{aligned}
	\end{equation}
	Let $s_j : D_j \hookrightarrow X$ be the canonical embedding and let $\pi_j : D_j \to \Spec \Qbb$ be the structural morphism. For each $ 1 \le j \le 20$, one has the following commutative diagram 
	\begin{equation}
		\xymatrix{H^1_\Mcal(\Spec \Qbb, \Qbb(3)) \ar[r]^{\pi_j^*} \ar[d]_{\reg^{1, 3}_\Qbb} & H^1_\Mcal(D_j, \Qbb(3)) \ar[r]^{{s_j}_*} \ar[d]_{\reg^{1, 3}_{D_j}} & H^3_\Mcal(X, \Qbb(4))_{\alg} \ar[d]_{\reg^{3, 4}_{X, \alg}}\\
			H^0(\mathrm{pt}, \Rbb(2))^+ \ar[r]& H^0(D_j(\Cbb), \Rbb(2))^+ \ar[r]& H^2(X(\Cbb), \Rbb(3))^+_{\alg},
		}
	\end{equation}
	where $\mathrm{pt}$ denotes a point, and the following sequence is in singular homology
	\begin{equation}
		\xymatrix{
			H_0(\mathrm{pt}, \Qbb)^+ & H_0(D_j(\Cbb), \Qbb)^+ \ar[l]^{\quad {\pi_j}_*} & H_2(X(\Cbb), \Qbb)^-_{\alg} \ar[l]^{s_j^*}}.
	\end{equation}
	One then considers the following pairings
	\begin{equation*}
		\left<\cdot, \cdot\right>_{X, \alg} :
		H_2(X(\Cbb), \Qbb)^-_{\alg} \times H^2(X(\Cbb), \Rbb(3))^+_{\alg}  \to \Rbb(3),
	\end{equation*}
	\begin{equation*}
		\left<\cdot, \cdot\right>_{D_j} :  H_0(D_j(\Cbb), \Qbb)^+ \times H^0(D_j(\Cbb), \Rbb(2))^+ \to \Rbb(2),
	\end{equation*}
	\begin{equation*}
		\left<\cdot, \cdot\right>_{\mathrm{pt}} :  H_0(\mathrm{pt}, \Qbb)^+ \times H^0(\mathrm{pt}, \Rbb(2))^+ \to \Rbb(2).
	\end{equation*}
	Now let $\xi \in H^3_\Mcal(X, \Qbb(4))_{\alg}$ be any nontrivial element. Using the decomposition \eqref{248}, we write 
	$$\xi  = \sum_{j=1}^{20} \xi_j, \quad \xi_j \in H^3_\Mcal(X, \Qbb(4))_{\alg, j}.$$
	Let $\gamma$ be any nontrivial element  in $H_2(X(\Cbb), \Qbb)^-_{\alg}$. By \eqref{250}, there exists $\alpha_j \in H^1_\Mcal(\Spec \Qbb, \Qbb(3))$  for $1 \le j \le 20$ such that
	\begin{equation}
		\left<\gamma, \reg^{3, 4}_{X}(\xi)\right>_{X, \alg} = \sum_{j=1}^{20} \left<\gamma, \reg^{3, 4}_{X}(\xi_j)\right>_{X, \alg} = \sum_{j=1}^{20} \left<\gamma, \reg^{3, 4}_{X} \circ s_{j*} \circ \pi_j^* (\alpha_j)\right>_{X, \alg}.
	\end{equation}
	We have
	\begin{equation}\label{251}
		\left<\gamma, s_{j*} \circ \pi_j^* \circ  \reg^{1, 3}_\Qbb (\alpha_j)\right>_{X, \alg} =  (2 \pi i)\cdot \left<s_j^* (\gamma), \pi_{j}^*  \circ \reg^{1, 3}_\Qbb (\alpha_j)\right>_{D_j} = (2 \pi i)\cdot \left<\pi_{j*} \circ s_j^* (\gamma), \reg^{1, 3}_\Qbb (\alpha_j)\right>_\mathrm{pt},
	\end{equation}
	where the two last equalities follow from the adjunction formula. Therefore, Beilinson's conjecture for the algebraic part of the motivic cohomology then reduces to Borel's theorem (\cite{Bor74}), which states that image of the motivic cohomology $H^1_\Mcal(\Spec \Qbb, \Qbb(3))$ under the Beilinson regulator map is related to the Riemann zeta value $\zeta(3)$ (see e.g., \cite[Section 3.3]{Tri24}). Therefore, we have
	\begin{equation}
		\left<\pi_{j*} \circ s_j^* (\gamma), \reg^{1, 3}_\Qbb (\alpha_j)\right>_\mathrm{pt} = a \cdot  (2\pi i)^2 \cdot \zeta'(-2), \quad a \in \Qbb.
	\end{equation}
	This implies that for any nontrivial element $\xi \in H^3_\Mcal(X, \Qbb(4))_{\alg}$ and $\gamma \in H_2(X(\Cbb), \Qbb)^-_{\alg}$, we have  
	\begin{equation}
		\dfrac{1}{(2\pi i)^3} \left<\gamma, \reg^{3, 4}_X(\xi)\right>_{X, \alg} = a \cdot \zeta'(-2), \quad a \in \Qbb.
	\end{equation}
\end{proof}

We arrive at the following direct corollary about the Mahler measure of a four-variable exact polynomial whose Maillot variety is a singular $K3$ surface with Picard rank 20 over $\Qbb$.
\begin{cor}\label{final}
Let $P (x, y, z, t) \in \Qbb[x, y, z, t]$ be a nonzero polynomial satisfying all the conditions in Theorem \ref{4varsmain}. Let $X$ be a smooth compactification of the Maillot variety $W_P$. Assume that $X$ is a singular $K3$ surface with Picard rank 20 over $\Qbb$. Then under Conjecture \ref{242}, we have
\begin{equation}
\m(P) = \m(\tilde{P}) + a\cdot L'(f, -1) + b \cdot \zeta'(-2), \quad a, b \in \Qbb,
\end{equation}
where $f$ is the newform defined in Theorem \ref{liv}.
\end{cor}

\begin{proof}
  Recall that under the conditions $\Gamma \subset V_P^\reg$ and $\partial \Gamma \subset Y(\Cbb)$, $\partial \Gamma$ defines an element in $H_2(X(\Cbb), \Qbb)^-.$ Using the following decomposition 
  $$H_2(X(\Cbb), \Qbb)^- = H_2(X(\Cbb), \Qbb)^-_{\tr} \oplus H_2(X(\Cbb), \Qbb)^-_{\alg},$$
  we get $[\partial \Gamma] = [\partial \Gamma]_{\tr} + [\partial\Gamma]_{\alg},$ where $[\partial \Gamma]_\epsilon \in H_2(X(\Cbb), \Qbb)^-_\epsilon$ for $\epsilon \in \{\tr, \alg\}$.
  Let $\Lambda$ be the element defined in \eqref{Lambdamain}. Under the conditions in Theorem \ref{4varsmain}, $\Lambda$ defines an element in $H^3_\Mcal(X, \Qbb(4))$ and 
   $$\m(P) - \m (\tilde{P}) = - \dfrac{1}{(2\pi i)^3} \left<[\partial \Gamma], \reg^{3, 4}_X(\Lambda)\right>_X.$$
  Now using the decomposition in Proposition \ref{370},
  $$H^3_\Mcal(X, \Qbb(4)) = H^3_\Mcal(X, \Qbb(4))_{\tr} \oplus H^3_\Mcal(X, \Qbb(4))_{\alg},$$
  we decompose $\Lambda$ into its transcendental and algebraic components
   $\Lambda = \Lambda_{\tr}+\Lambda_{\alg},$
   where $\Lambda_{\epsilon} \in H^3_\Mcal(X, \Qbb(4))_{\epsilon}$ for $\epsilon \in \{\tr, \alg\}$. Since the transcendental lattice is the orthogonal complement of the Néron-Severi lattice in the cohomology group, we have 
   $$\m(P) - \m(\tilde{P}) = -\dfrac{1}{(2 \pi i)^3} \left<[\partial \Gamma]_{\tr}, \reg^{3, 4}_X(\Lambda_{\tr})\right>_{X, \tr} -\dfrac{1}{(2 \pi i)^3} \left<[\partial \Gamma]_{\alg}, \reg^{3, 4}_X(\Lambda_{\alg})\right>_{X, \alg}.$$
   Under Conjecture \ref{242}, we have 
   $$-\dfrac{1}{(2 \pi i)^3} \left<[\partial \Gamma]_{\tr}, \reg^{3, 4}_X(\Lambda_{\tr})\right>_{X, \tr} 
 = a \cdot L'(f, -1), \quad a \in \Qbb,$$
 where $f$ is the new form defined in Theorem \ref{liv}. 
  And by Proposition \ref{alg}, one gets 
  $$-\dfrac{1}{(2 \pi i)^3} \left<[\partial \Gamma]_{\alg}, \reg^{3, 4}_X(\Lambda_{\alg})\right>_{X, \alg} = a \cdot \zeta'(-2), \quad a \in \Qbb.$$
  Therefore, under Conjecture \ref{242}, we have 
  $$\m(P) = m(\tilde{P}) + a \cdot L'(f, -1) + b \cdot \zeta'(-2), \quad a, b \in \Qbb.$$
\end{proof}

	\section{Application}\label{ex4var}

	The main aim of this section is to prove that (under Goncharov's and Beilinson's conjecture) the Mahler measure of the polynomial $P = (x+1)(y+1)(z+1) +t$ is a rational linear combination of a special $L$-value of the unique CM new form of weight 3, level 7 and a Riemann zeta value
	\begin{equation}\label{Fconj}
	    \m(P) = a \cdot L'(f_7, -1) + b \cdot  \zeta'(-2), \quad a, b \in \Qbb.
	\end{equation}
    This identity is numerically conjectured by Brunault that $a = -6$ and $b = -48/7$ (see \cite{Bru23}). Before going into the details of the proof, we present some preliminary results on elliptic modular surfaces, which will be used in Section \ref{application}.

\subsection{Elliptic modular $K3$ surfaces}\label{elliptic modular K3 surface}
In this section, we give a necessary and effective criterion for an elliptic modular surface to be an elliptic 
$K3$ surface. It is by no means a new result, but since we have been unable to find a proof in the literature, we include one here for the reader’s convenience.

Let us recall the definition of elliptic surfaces. Let $C$ be a complex smooth projective curve. An elliptic surface over $C$ is a smooth projective surface $X$ with a surjective morphism $\pi : X \to C$  such that 
	\begin{itemize}
		\item  For all $s \in C$, $\pi^{-1}(s)$  does not contain any curve $E$ satisfying that $E \simeq \Pbb^1$ and $E^2 = -1$.
        
		\item Let $\eta$ be the generic point of $C$. The generic fiber $X_{\eta} := \pi^{-1}(\eta)$ is an elliptic curve over $K := k(C)$.
		
		\item  There exists $s \in C$ such that $\pi^{-1}(s)$ is singular. 
	\end{itemize}

Let $\Hcal = \{\tau \in \Cbb : \Im(\tau) >0\}$ be the upper half-plane. The usual action of $\SL_2(\Zbb)$ on $\Hcal$ is given by 
\begin{equation*}
    \gamma \cdot \tau = \dfrac{a\tau + b}{c\tau + d}, \quad \text{where } \gamma = \begin{bmatrix}
        a & b \\ c & d
     \end{bmatrix} \in \SL_2(\Zbb).
\end{equation*}
For an integer $N \ge 1$, one defines 
    \begin{equation*}
        \Gamma(N)  = \left\{ \gamma \in \SL_2(\Zbb): \gamma \equiv \begin{bmatrix}
     1 & 0\\0 & 1 \end{bmatrix} \mod N \right\}, \quad \Gamma_1(N) = \left\{\gamma \in \SL_2(\Zbb): \gamma \equiv \begin{bmatrix}
     1 & \ast\\0 & 1 \end{bmatrix} \mod N\right\}.
      \end{equation*}
    A congruence subgroup of $\SL_2(\Zbb)$ is a subgroup of $\SL_2(\Zbb)$ containing $\Gamma(N)$ for some integer $N \ge 1$.  The modular curve associated to a congruence subgroup $\Gamma \in \SL_2(\Zbb)$ is defined to be the Riemann surface $\Gamma \setminus \Hcal$, denoted by $C_\Gamma^0$.  This Riemann surface can be compactified by adding the finite set of points $\Gamma \setminus \Pbb^1(\Qbb)$. One denotes by $C_\Gamma$ this compactification and the finite set of points added is called the set of cusps of $C_\Gamma$. Recall that if $\Gamma = \Gamma_1(N)$, then $C^0_\Gamma$ and $C_\Gamma$ are defined over $\Qbb$ (see \cite[Section 6.7]{Shi71}).

We set $\Ecal := \Zbb^2 \setminus (\Hcal \times \Cbb),$ where $\Zbb^2$ acts on $\Hcal \times \Cbb$ as follows
\begin{equation}
    (m, n) \cdot (\tau; z) := (\tau; z + m\tau +n), \quad \text{ for } m, n \in \Zbb, \tau\in \Hcal, z \in \Cbb.
\end{equation}
The space $\Ecal$ is a complex analytic manifold of dimension 2 and we have a holomorphic projection 
$$\Ecal \to \Hcal, \quad (\tau; z) \mapsto \tau,$$
such that for each $\tau \in \Hcal$, the fiber $\Ecal_\tau =  \Zbb^2 \setminus (\tau \times \Cbb) \simeq \Cbb / (\Zbb + \tau \Zbb)$.  The action of $\SL_2(\Zbb)$ on $\Ecal$ is given by
\begin{equation*}
    \gamma \cdot (\tau; z) = \left(\dfrac{a \tau + b}{c \tau + d}; \dfrac{z}{c \tau + d}\right) \quad \text{ for } \gamma = \begin{bmatrix}
        a & b\\c &d
    \end{bmatrix} \in \SL_2(\Zbb).
\end{equation*}
This action is not free, so it is necessary to consider congruence subgroups $\Gamma$ that act freely on $\Hcal$.  One then calls the quotient $\Gamma\setminus \Ecal$ together with the canonical projection $\Gamma\setminus \Ecal \to \Gamma \setminus \Hcal$ the \textit{universal elliptic curve associated to $\Gamma$}. We denote it by $\Ecal_\Gamma$. Recall that universal elliptic curves are algebraic varieties.

 \begin{definition}[\textbf{The elliptic modular surfaces associated to $\Gamma$}]\label{defellipticmodular}
 	Let $\Gamma$ be a congruence subgroup of $\SL_2(\Zbb)$ acting freely on $\Hcal$.  One calls the unique elliptic surface $X_\Gamma \to C_\Gamma$ associated to the universal elliptic curve $\Ecal_\Gamma \to C_\Gamma^0$ the elliptic modular surface associated to $\Gamma$ (see e.g., \cite[Theorem 5.19]{SS19}).
 \end{definition}

 \begin{lm}\label{266}
 	Let $f: X_\Gamma \to C_\Gamma$ be an elliptic modular surface. We have 
 	\begin{equation}\label{2isos}
 		H^0(C_\Gamma, \Omega_{C_\Gamma}^1) \simeq \Scal_2(\Gamma), \quad H^0 (X_\Gamma, \omega_{X_{\Gamma}}) \simeq \Scal_3(\Gamma),
 	\end{equation}
 	where $\Scal_j(\Gamma)$ is the complex vector space of cusp forms of weight $j$ and level $\Gamma$.
 \end{lm}

 \begin{proof}
 	Denote by $\omega_{X_\Gamma/C_\Gamma}$ the relative dualizing sheaf of $f$.  As $X_\Gamma$ is a smooth projective surface and $f$ is a projective flat morphism, we have
 	\begin{equation}\label{canonialsheaf}
 		\omega_{X_\Gamma/C_\Gamma} = \omega_{X_\Gamma} \otimes f^*(\omega_{C_\Gamma}^{-1}),
 	\end{equation} 
 	where $\omega_{X_\Gamma}$ is the canonical sheaf of $X_{\Gamma}$ and $\omega_{C_\Gamma}$ is the canonical sheaf of $C_\Gamma$. Denote by $\omega = f_* (\omega_{X_\Gamma/C_\Gamma})$. From formula \eqref{canonialsheaf}, we have 
 	\begin{equation}\label{canonicalsheafXGamma}
 		\omega_{X_\Gamma} = f^*(\omega \otimes \omega_{C_\Gamma}) = f^*(\omega \otimes \Omega_{C_{\Gamma}}^1).
 	\end{equation}
 	We recall the following natural isomorphism, which is called the \textit{Shimura isomorphism} (see \cite{Shi59})
 	\begin{equation}\label{Shimuraiso}
 		H^0(C_\Gamma, \omega^j \otimes \Omega_{C_\Gamma}^1) \simeq \Scal_{j+1}(\Gamma),
 	\end{equation}
 	for $j \ge 1$. This implies that $H^0(C_\Gamma, \Omega_{C_\Gamma}^1) \simeq \Scal_2(\Gamma)$ and 
 	$H^0 (X_\Gamma, \omega_{X_{\Gamma}}) \simeq H^0(C_\Gamma, \omega \otimes \Omega_{C_\Gamma}^1)\simeq \Scal_3(\Gamma).$
 \end{proof}

\begin{definition}[\textbf{Elliptic $K3$ surfaces}]
    An elliptic $K3$ surface is  an elliptic surface $X \to \Pbb^1$ where $X$ is an algebraic $K3$ surface.
\end{definition}
 
The following result gives us conditions when an elliptic modular surface is an elliptic $K3$ surface. 

\begin{lm}\label{modularK3}
	Let $\Gamma$ be a congruence subgroup of $\SL_2(\Zbb)$ acting freely on the upper half-plane $\Hcal$ and  let $f : X_\Gamma \to C_\Gamma$ be the associated elliptic modular surface. Then $f$ is  an elliptic $K3$ surface if and only if 
	$$\dim \Scal_2(\Gamma) = 0, \text{ and } \dim \Scal_3(\Gamma) = 1,$$ where $\Scal_j(\Gamma)$ is the complex vector space of cusp forms of weight $j$, level $\Gamma$.
\end{lm}

\begin{proof}
	Suppose that $f : X_\Gamma \to C_\Gamma$ is an elliptic $K3$ surface. In particular, we have $C_\Gamma \simeq \Pbb^1$. Then by formula \eqref{2isos}, we have $\dim \Scal_2(\Gamma) = \dim H^0(C_\Gamma, \Omega_{C_\Gamma}^1) = 0.$ As $X_\Gamma$ is an algebraic $K3$ surface, by formula \eqref{2isos}, we also have 
	\begin{equation*}
		\dim \Scal_3(\Gamma) = \dim H^0(X_\Gamma, \omega_{X_\Gamma}) = \dim H^0(X_\Gamma, \Omega_{X_\Gamma}^2) = 1.
	\end{equation*}
	Now we prove the converse. Again by formula \eqref{2isos}, $\dim \Scal_2(\Gamma) = 0$ implies that 
	$H^0(C_\Gamma, \Omega_{C_\Gamma}^1) = 0.$ Therefore, $C_\Gamma$ is a smooth projective curve of genus 0, hence $C_\Gamma \simeq \Pbb^1$.   Let
	$$\omega = f_* (\omega_{X_\Gamma/C_\Gamma}),$$ where $\omega_{X_\Gamma/C_\Gamma}$ is the relative dualizing sheaf of $f$. By definition of elliptic surfaces, $f$ is not isotrivial,  so we can write
	$\omega = \Ocal_{\Pbb^1} (a)$ for some suitable positive integer $a$. The positivity of $a$ follows from the positivity of the direct image of the relative canonical sheaf.
	From the formula \eqref{canonialsheaf}, we then have $$\omega_{X_\Gamma} = f^*(\omega \otimes \omega_{C_\Gamma}) = f^*(\Ocal_{\Pbb^1}(a) \otimes \Omega_{\Pbb^1}^1)= f^*(\Ocal_{\Pbb^1}(a-2)).$$
	By Shimura's isomorphism \eqref{Shimuraiso}, we have
	\begin{equation*}
		1 = \dim \Scal_3(\Gamma) = \dim H^0(C_\Gamma, \omega \otimes \Omega_{C_\Gamma}^1) = \dim H^0 (\Pbb^1, \Ocal_{\Pbb^1}(a-2)).
	\end{equation*}
	Therefore, $a = 2$. Hence $\omega_{X_\Gamma} =  f^*(\Ocal_{\Pbb^1}) = \Ocal_{X_\Gamma}$. We consider Leray spectral sequence 
	\begin{equation*}
		E_2^{p, q} = H^p(C_\Gamma, R^qf_* \Ocal_{X_\Gamma}) \Longrightarrow H^{p+q}(X_\Gamma, \Ocal_{X_\Gamma}),
	\end{equation*}
	which degenerates at $E_2$-term (see e.g., \cite[Section 1.3]{PS08}). Since $R^1 f_* \Ocal_{X_\Gamma} \simeq \omega^\vee  \simeq \Ocal_{\Pbb^1_k}(-2)$, one obtains that  $$Gr^0 H^1 (X_\Gamma, \Ocal_{X_\Gamma}) = E_\infty^{0, 1} = E_2^{0, 1} = H^0(X_\Gamma, R^1 f_* \Ocal_{X_\Gamma}) \simeq H^0 (\Pbb^1_k, \Ocal_{\Pbb^1_k}(-2)) = 0,$$
	and $Gr^2 H^1(X_\Gamma, \Ocal_{X_\Gamma}) = E_\infty^{2, -1} = 0.$
	This implies that 
	$$H^1(X_\Gamma, \Ocal_{X_\Gamma}) = Gr^1 H^1(X_\Gamma, \Ocal_{X_\Gamma}) = E_\infty^{1, 0} = E_2^{1, 0} = H^1(C_\Gamma, f_*\Ocal_{X_\Gamma}).$$
	As $f : X_\Gamma \to C_\Gamma$ is a projective morphism and $C_\Gamma \simeq \Pbb^1$ is normal, we have $f_*\Ocal_{X_\Gamma} = \Ocal_{\Pbb^1}$ (see e.g., \cite[proof of Corollary III.11.4]{Har77}). Therefore, we have 
	$H^1(C_\Gamma, f_*\Ocal_{X_\Gamma})  \simeq H^1(\Pbb^1, \Ocal_{\Pbb^1}) = 0.$
\end{proof}

\begin{ex}[\textbf{The elliptic modular surface associated to $\Gamma_1(7)$}]\label{K3modularsurface7}
Denote by $\Gamma :=  \Gamma_1(7)$. Recall that the modular curves $C_{\Gamma}^0$ and $C_\Gamma$ are algebraic curves defined over $\Qbb$. Denote by $K := \Qbb(t)$ the function field of $C_\Gamma$. We consider the following Weierstrass equation
	\begin{equation}\label{E17}
		E :  V^2 + (1+t-t^2)UV + (t^2 -t^3)V = U^3 + (t^2 -t^3)V^2,
	\end{equation}
	which defines an elliptic curve over $K$ if and only if $t \notin \Sigma := \{0, 1, \infty, \text{ the roots of } t^3 - 8t^2 + 5t +1\}.$  One then can check that $P = (0,0)$ is a point of order $7$. Denote by $\Ecal \subset \Pbb^2_{[U:V]} \times C_\Gamma^0$ the open surface defined by  \eqref{E17}. It can be shown that there is an isomorphism 
	$C_\Gamma^0 \simeq \Pbb^1 (\Cbb) \setminus \Sigma$
	such that the canonical projection  
	$\Ecal \to C_\Gamma^0$ is the universal elliptic curve associated to $\Gamma$. Denote by $\pi : X \to C_\Gamma$ the corresponding elliptic modular surface.
Since $\dim \Scal_2(\Gamma) = 0$ and $\dim \Scal_3(\Gamma) = 1$, by Lemma \ref{modularK3}, $X$ is an algebraic $K3$ surface defined over $\Qbb$. The singular fibers of $\pi$ are of Kodaira types $I_7, I_7, I_7, I_1, I_1, I_1$.  Hence the Picard rank is  \begin{equation*}
		\rho(X_{\bar \Qbb}) = r + 2  
 +\sum_{t \in \Sigma}(m_t - 1) = r + 2 + 6 + 6 + 6 = r + 20,
	\end{equation*}
where $r$ is the rank of the Mordell-Weil group $E(K)$ and $m_t$ is the number of irreducible components of the singular fiber $X_t := \pi^{-1}(t)$. As $X$ is a $K3$ surface, we have $\rho(X_{\bar \Qbb}) \le 20$. This implies that $r = 0$ and  $\rho(X_{\bar \Qbb}) = 20$, therefore, $X$ is a singular $K3$ surface. The Néron-Severi group $\NS(X_{\bar \Qbb})$ has a $\Qbb$-basis given by the divisor classes of
	\begin{equation}\label{270}
		\{\bar O; \ F; \Theta_{s, i}, 1 \le i \le m_s - 1, s \in \Sigma\}, 
	\end{equation}
	where $\bar O$ is the closure of the trivial point $O$ of $E(K)$ in $X$, $F$ is a general fiber, and $\Theta_{s, i}$ for $1 \le i \le m_s-1$ are the irreducible components of the singular fiber $X_s$  which do not intersect with $\bar O$ (see \cite[Proposition 6.3]{SS19}). Recall that, for any point $Q \in E(K)$, its closure in $X$ is a curve isomorphic to $C_\Gamma$, hence it is defined over $\Qbb$. It follows from the Weierstrass equation \eqref{E17} that a general fiber is an elliptic curve defined over $\Qbb$. The singular fiber $I_1$ is irreducible, hence does not contribute any generator to the basis \eqref{270} of $\NS(X_{\bar \Qbb})$. Since the point $P$ intersects each $I_7$ fiber at a different nontrivial component, all fiber irreducible components of each $I_7$ fiber are defined over $\Qbb$. Therefore, $X$ is a singular $K3$ surface with Picard rank $20$ over $\Qbb$. 
\end{ex}

	\subsection{The Mahler measure of the polynomial $P = (x+1)(y+1)(z+1) + t$}\label{application}
In this section, we prove Mahler measure identity \eqref{Fconj} under certain conjectures. In particular, we show that the Maillot variety $W_P$ has a regular model which is isomorphic to the $K3$ surface associated to the congruence subgroup $\Gamma_1(7)$.

Recall that the Maillot variety $W_P$ is the intersection $V_P \cap V_{P^*}$ where $P^* := P(1/x, 1/y, 1/z, 1/t)$. By eliminating $t$, we see that $W_P$ is isomorphic to the surface in $(\Cbb^\times)^3$ given by
	\begin{equation}\label{WP}
		x y z - (x+1)^2 (y+1)^2 (z+1)^2 = 0.
	\end{equation}
     By using the transformation $x \mapsto x$, $y \mapsto -y$, $z \mapsto z+1$, the equation \eqref{WP} becomes
\begin{equation*}
    -xy(z-1) - (x+1)^2 (-y+1)^2 z^2 = 0,
\end{equation*}
and the tangent cone at the singularity $(0, 0, 0)$ is given by $xy - z^2 = 0$. We then have the singular locus of $W_P$ only contains singularities in codimension 2, which are $A_1$-singularities $(-1, 0, 0),$ $ (0, -1, 0)$, and $(0, 0, -1)$. Let $\overline{W_P}$ be the compactification of $W_P$ in $\Pbb^1_{[x:x']} \times \Pbb^1_{[y:y']} \times \Pbb^1_{[z:z']}$. It is given by the following equation 
\begin{equation}\label{WPbar}
    x y z x' y' z' - (x+x')^2 (y+y')^2 (z+z')^2 = 0.
\end{equation}
It has 8 affine charts and each of them is defined by the same equation as \eqref{WP} in the corresponding affine coordinates. Therefore, the singular locus of $\overline{W_P}$ only consists of  finite $A_1$-singularities as well. In particular, $\overline{W_P}$ is regular in codimension 1,  hence normal.  Denote by 
	\begin{equation}\label{X}
		\pi: X \to \overline{W_P}
	\end{equation} the blowing up of $\overline{W_P}$ along  all these singularities. As all the singularities are of type $A_1$, $X$ is nonsingular. Moreover, $X$ is projective since $\overline{W_P}$ is projective. In fact, $f$ is a birational projective morphism (see \cite[Propostition II.7.16]{Har77}). It suffices to look at the blowing-up of the affine chart $W_P$ along its singularities. The proper transform $\widetilde{W_P}$ of $W_P$  is contained in the weighted projective space $\Pbb(0, 0, 0, 1, 1, 1, 1)$ with corresponding coordinates 
	$((x_1, x_2, x_3), (x_4: x_5: x_6: x_7))$. And $\widetilde{W_P}$ is defined over $\Qbb$  by the following equations
	\begin{equation}\label{21eqs}
		\begin{cases}
			x_1 x_2 x_3 - (x_1+1)^2 (x_2+1)^2 (x_3+1)^2 = 0,\\
			\text{20 other equations with $\Qbb$-coefficients given by blowing up } \Abb^3 \text{ along the singular locus} ,
		\end{cases}
	\end{equation}
	together with the projection map
	\begin{equation}\label{blowupmap}
		\tilde \pi: \widetilde{W_P} \to W_P, \quad ((x_1, x_2, x_3), (x_4: x_5: x_6: x_7)) \mapsto (x_1, x_2, x_3).
	\end{equation}
	Then the map $\pi: X\to \overline{W_P}$ is obtained by gluing such blowing up of affine charts. Using Magma, one can easily obtain that the exceptional curve at the singular point $(-1, 0, 0)$ is the smooth curve given by the following equation in the weighted projective space $\Pbb(0, 0, 0, 1, 1, 1, 1)$
	\begin{equation*}
		E_1: \begin{cases}
			x_1 + 1 = 0,\\
			x_2 = 0,\\
			x_3 = 0,\\
			x_6 = 0,\\
			x_4^2 - 2x_4 x_5 + x_5^2 - 2x_4 x_7 + 3 x_5 x_7 + x_7^2 = 0.
		\end{cases}
	\end{equation*}
	Similarly, the exceptional curves above $(0,-1,0)$ and $(0, 0, -1)$ are respectively given by 
	\begin{equation*}
		E_2: \begin{cases}
			x_1 = 0,\\
			x_2 + 1 = 0,\\
			x_3 = 0,\\
			x_5 + x_6 = 0,\\
			x_4 x_6 + x_6^2 + x_6 x_7 - x_7^2 = 0,
		\end{cases} \quad E_3: \begin{cases}
			x_1 = 0,\\
			x_2 = 0,\\
			x_3 + 1 = 0,\\
			x_6 + x_7 = 0\\
			x_5^2 + x_4 x_7 + x_5 x_7 - x_7^2 = 0.
		\end{cases}
	\end{equation*}
	We have the following result. 
	\begin{prop}\label{k3}
		The surface $X$ defined in \eqref{X} is an algebraic $K3$ surface over $\Qbb$.
	\end{prop}
	
	\begin{proof}
		Let $\pr_i : \Pbb^1 \times \Pbb^1 \times \Pbb^1 \to \Pbb^1$ be the $i$-th canonical projection.  We have  
		\begin{align*}
			\omega_{\Pbb^1 \times \Pbb^1 \times \Pbb^1} &= \pr_1^*\omega_{\Pbb^1} \otimes \pr_2^* \omega_{\Pbb^1} \otimes \pr_3^* \omega_{\Pbb^1}\\
			&= \pr_1^* \Ocal_{\Pbb^1} (-2)  \otimes \pr_2^* \Ocal_{\Pbb^1} (-2) \otimes \pr_3^* \Ocal_{\Pbb^1} (-2),
		\end{align*}
		where the first equality follows from the formula of canonical sheaf on product schemes (see \cite[Exercise II.8.3]{Har77}). Recall that the dualizing complex of $\Pbb^1 \times \Pbb^1 \times \Pbb^1$ is given by
		\begin{equation*}
			\omega_{\Pbb^1 \times \Pbb^1 \times \Pbb^1}^\bullet  \simeq \left(\pr_1^*\Ocal_{\Pbb^1}(-2)\otimes\pr_2^*\Ocal_{\Pbb^1}(-2)\otimes\pr_3^*\Ocal_{\Pbb^1}(-2)\right)[3].
		\end{equation*} 
		Let us denote by $D$ the Cartier divisor corresponding to $\overline{W_P}$ in $P^1 \times \Pbb^1 \times \Pbb^1$  and $i :  \overline{W_P} \hookrightarrow \Pbb^1 \times \Pbb^1 \times \Pbb^1$ be the canonical embedding. Recall that the invertible sheaf associated to $D$ is given by 
		$$\Ocal_{\Pbb^1
			\times \Pbb^1 \times \Pbb^1}(D) = \pr_1^*\Ocal_{\Pbb^1}(2)\otimes \pr_2^*\Ocal_{\Pbb^1}(2)\otimes \pr_3^*\Ocal_{\Pbb^1}(2).$$
  We then have 
		\begin{align*}
			\omega_{D}^\bullet &\simeq i^*(\omega_{\Pbb^1\times \Pbb^1 \times \Pbb^1}^\bullet) \otimes i^*\Ocal_{\Pbb^1 \times \Pbb^1 \times \Pbb^1}(D)[-1]\\
			&\simeq i^*(\pr_1^* \Ocal_{\Pbb^1} \otimes \pr_2^* \Ocal_{\Pbb^1} \otimes \pr_3^* \Ocal_{\Pbb^1})[2]\\
			&\simeq i^* \Ocal_{\Pbb^1 \times \Pbb^1 \times \Pbb^1}[2]\\
			&\simeq \Ocal_{D}[2],
		\end{align*}
		where the first isomorphism follows from \cite[Lemma 10.2.3]{Tri24}. Hence we have $\omega_D \simeq \Ocal_{D}$.  Therefore, we obtain that
		$$\omega_X \simeq \pi^*\omega_{D} \simeq \pi^*\Ocal_{D} \simeq \Ocal_X,$$
        where the first isomorphism is due to $f : X \to D$ being the blowing up of $D$ along $A_1$-singularities.
	In order to compute $H^1(X, \Ocal_X)$, we consider the Leray spectral sequence
		\begin{equation*}
			E^{p, q}_2 = H^p(D, R^q \pi_* \Ocal_X) \Longrightarrow H^{p+q}(X, \Ocal_X),
		\end{equation*}
		which degenerates at the $E_2$-term (see e.g., \cite[Section 1.3]{PS08}).
		Notice that $R^1\pi_*\Ocal_X = 0$ as $A_1$-singularities are rational singularities. Therefore, we have  $$\mathrm{Gr}^0 H^1 (X, \Ocal_X) = E_\infty^{0,1} = E_2^{0,1} = H^0(D, R^1 \pi_*\Ocal_X) = 0.$$ We also have $\mathrm{Gr}^2 H^1 (X, \Ocal_X) = E_\infty^{2, -1} = 0$. Hence 
		$$H^1 (X, \Ocal_X) = \mathrm{Gr}^1 H^1 (X, \Ocal_X) = E_\infty^{1,0} = E_2^{1,0} =  H^1 (D, \pi_*\Ocal_X) = H^1(D, \Ocal_{D}),$$ 
		where the last equality follows from the fact that $\pi_*\Ocal_X \simeq \Ocal_{D}$ as $\pi$ is a birational projective morphism and that $D$ is normal  (see e.g., \cite[proof of Corollary III.11.4]{Har77}). Therefore, it is sufficient to show that $H^1(D, \Ocal_D) = 0$. By Künneth's formula, one can show that $$H^1(\Pbb^1\times \Pbb^1\times \Pbb^1,\Ocal_{\Pbb^1\times \Pbb^1\times \Pbb^1}) = 0, \quad H^2(\Pbb^1\times \Pbb^1\times \Pbb^1, \Ocal_{\Pbb^1\times \Pbb^1\times \Pbb^1}(-D)) = 0.$$ Recall that as $D$ is effective, $\Ocal_X(-D)$ 
		is the sheaf of ideals defining $D$. We then have the following short exact sequence
		$$0\to \Ocal_{\Pbb^1\times \Pbb^1\times \Pbb^1}(-D) \to \Ocal_{\Pbb^1\times \Pbb^1\times \Pbb^1}\to \Ocal_D\to 0.$$ It induces  the following exact sequence
		$$H^1(\Pbb^1\times \Pbb^1\times \Pbb^1,\Ocal_{\Pbb^1\times \Pbb^1\times \Pbb^1})\to H^1(D,\Ocal_D)\to H^2(\Pbb^1\times \Pbb^1\times \Pbb^1,\Ocal_{\Pbb^1\times \Pbb^1\times \Pbb^1}(-D)).$$
		This implies that $H^1(D,\Ocal_D)=0.$
	\end{proof}

	\begin{prop}\label{234}
		The surface $X$ is a singular $K3$ surface with Picard rank 20 over $\Qbb$.
	\end{prop}
	
	\begin{proof}
		By \cite{Lec23}, under the following change of variables
		\begin{equation*}
			x = - \dfrac{(dv-1)^2 (du-1)}{d^2 v (v-1)^2 (u-1)^2}, \quad y = du - 1, \quad z= -v,
		\end{equation*}
		with the inverse 
		\begin{equation*}
			u = \dfrac{(y + 1)(y + z + 1)(x + 1)}{(z + 1)(y + 1)(x + 1) + y}, \quad v= -z, \quad d = \dfrac{(z + 1)(y + 1)(x + 1) + y}{(y + z + 1)(x + 1)},
		\end{equation*}
		we obtain that $W_P$ is birational to the following surface
		\begin{equation*}
			(uv-u-v)(du-1)(dv-1) = d(d-1) uv(u-1)(v-1).
		\end{equation*}
		Then, by the following change of variables
		\begin{equation*}
			u = \dfrac{V}{U^2}, \quad v  = - \dfrac{U+d^2-d^3}{V}, \quad d = d,
		\end{equation*}
		with the inverse
		\begin{equation*}
			U = \dfrac{-d(d-1)}{uv-u-v}, \quad V = u \left(\dfrac{d(d-1)}{uv-u-v}\right)^2,
		\end{equation*}
		we obtain the following Weierstrass equation
		\begin{equation*}
			V^2 + (1+d-d^2)UV + (d^2 -d^3)V = U^3 + (d^2 -d^3)U^2,
		\end{equation*}
        which defines the $K3$ surface associated to the congruence subgroup $\Gamma_1(7)$ (see Example \ref{K3modularsurface7})
		Recall that two birational $K3$ surfaces are isomorphic since algebraic $K3$ surfaces are minimal (see e.g., \cite[Chapter VII]{Bea96}). Hence by Proposition \ref{k3},  $X$ is isomorphic to the $K3$ surface associated to $\Gamma_1(7)$. Therefore, $X$ is a singular $K3$ surface with Picard rank 20 over $\Qbb$.
	\end{proof}

	Now let us construct an element in $H^3(\Gamma(X, 4))$, where $\Gamma(X, 4)$ is Goncharov's polylogarithmic complex \eqref{bicomplex4}. We have the following decomposition in $\bigwedge^4 \Qbb(V_P)^\times_\Qbb$
	\begin{equation*}
		x \wedge y \wedge z \wedge t = x \wedge (1+x) \wedge y \wedge z  - y \wedge (1+y) \wedge x \wedge z + z \wedge (1+z) \wedge x \wedge y.
	\end{equation*}
	Denote by $F = \Qbb(W_P)$ the function field of $W_P$. As in Definition \ref{310}, we define the following elements in $\Bcal_2(F) \otimes \bigwedge^2 F_\Qbb^\times$
	\begin{equation}\label{xidecomp}
		\xi  = \{x\}_2 \otimes y \wedge z - \{y\}_2 \otimes x \wedge z + \{z\}_2 \otimes x \wedge y,
	\end{equation}
	and
	\begin{equation*}
		\xi^* = \left\{\frac{1}{x}\right\}_2 \otimes \frac{1}{y} \wedge \frac{1}{z} - \left\{\frac{1}{y}\right\}_2 \otimes \frac{1}{x} \wedge \frac{1}{z} + \left\{\frac{1}{z}\right\}_2 \otimes \frac{1}{x} \wedge \frac{1}{y} = -\xi.
	\end{equation*}
	By Lemma \ref{Hn-1Fn}, $\xi$ defines a class in $H^3(\Gamma(F, 4))$. Let us compute the residues $\partial^{4,3}_p(\xi)$ for all $p\in X^1$ in the bicomplex \eqref{bicomplex4}.
	
	\begin{lm}
		The residues $\partial_p^{4, 3}(\xi)$ vanish for all $p \in X^1$.
	\end{lm}
	\begin{proof}
		Recall that locally $X$ is given by equations \eqref{21eqs} in the weighted projective $\Pbb(0, 0, 0, 1, 1, 1, 1)$ with the projection map 
		$$\pi: X \to \overline{W_P}, \quad ((x_1, x_2, x_3), (x_4 : x_5 : x_6: x_7)) \mapsto (x_1, x_2, x_3).$$ Recall that the residue map 
		\begin{equation*}
			\partial^{4,3} : B_2(F) \otimes \bigwedge^2 F^\times_\Qbb \to \bigoplus_{p \in X^1} B_2 (\Qbb(p)) \otimes \Qbb(p)^\times_\Qbb,
		\end{equation*}
		is defined by
		\begin{equation*}
			\{f\}_2 \otimes g \wedge h 
			\mapsto \left(\{f(p)\}_2 \otimes T_p\{g, h\}\right)_{p \in X^1},\quad \text{for } f, g, h \in F^\times, 
		\end{equation*}
        where the Tame symbol at $p \in X^1$  is given by \begin{equation*}
		T_p\{f, g\} := (-1)^{\ord_p(f) \ord_p(g)}\left(\dfrac{f^{\ord_p(g)}}{g^{\ord_p(f)}}\right)(p).
	\end{equation*}
	As the decomposition \eqref{xidecomp} of $\xi$, the residues of $\xi$ vanish outside the set of zeros and poles of $x_1, x_2, x_3$ on $X$. By the symmetry of $W_P$ in variables $x_1, x_2, x_3$, it suffices to show that the residues of $\xi$ at the zeros and poles of $x_1$ all vanish. By using Gröbner basis, we find that the zeros of $x_1$ in the affine chart $x_j \neq 0$  for $j = 2, \dots, 7$ are the curves $L_i^{(j)}$ given as follows:
    
		\begin{align*}
        &L^{(2)}_1 : \begin{cases}
            x_1 = 0,\\
            x_3 + 1 = 0,\\
            2x_4 - x_7 = 0,\\
            x_5 = 0,\\
            2x_6 + x_7 = 0,\\
        \end{cases}
        & &L^{(3)}_1 : \begin{cases}
            x_1 = 0,\\
            x_2 + 1 = 0,\\
            x_4 + x_6 = 0,\\
            x_5 + 2x_6 = 0\\
            x_7 = 0,\\
        \end{cases}\\
			&L^{(4)}_1 : \begin{cases}
				x_1 = 0,\\
				x_2 - x_7 + 1 = 0,\\
				x_3 + 1 = 0,\\
				x_5 = 0,\\
				x_6 + 1 = 0, 
			\end{cases}
			& &L^{(4)}_2 : \begin{cases}
				x_1 = 0,\\
				x_2 + 1 = 0,\\
				x_3 - x_5 + 1 = 0, \\
				x_6 + 1 = 0,\\
				x_7 = 0,
			\end{cases}\\
			&L^{(4)}_3 : \begin{cases}
				x_1 = 0,\\
				x_2 + 1 = 0, \\
				x_3 = 0,\\
				x_5 + x_6 = 0, \\
				x_6^2 + x_6x7 + x_6 - x_7^2 = 0,
			\end{cases}
			& &L^{(4)}_4 : \begin{cases}
				x_1 = 0,\\
				x_2 = 0,\\
				x_3 + 1 = 0,\\
				x_5^2 + x_5x_7 - x_7^2 + x_7 = 0,\\
				x_6 + x_7 = 0,
			\end{cases}
		\end{align*}
		\begin{align*}
			L^{(5)}_1 : \begin{cases}
				x_1 = 0,\\
				x_2 + 1 = 0,\\
				x_3 x_6 + x_6 + 1 = 0,\\
				x_4 + x_6 = 0,\\
				x_7 = 0
			\end{cases} 
			\ L^{(5)}_2 : \begin{cases}
				x_1 = 0,\\
				x_2 + 1 = 0,\\
				x_3 = 0,\\
				x_4 + x_7^2 + x_7 - 1 = 0,\\
				x_6 + 1 = 0,\\
			\end{cases}
			\ L^{(5)}_3 : \begin{cases}
				x_1 = 0,\\
				x_2 = 0,\\
				x_3 + 1 = 0,\\
				x_4 x_7 - x_7^2 + x_7 + 1 = 0,\\
				x_6 + x_7 = 0,
			\end{cases}
		\end{align*}
		\begin{align*}
			&L^{(6)}_1 : \begin{cases}
				x_1=0,\\
				x_2 + x_7 + 1=0,\\
				x_3 + 1=0,\\
				x_4 + 1=0,\\
				x_5=0,
			\end{cases}
			& &L^{(6)}_2 :\begin{cases}
				x_1=0,\\
				x_2 + 1=0,\\
				x_3 + x_5 + 1 =0,\\
				x_4 + 1 = 0,\\
				x_7 = 0,
			\end{cases}\\
			&L^{(6)}_3 : \begin{cases}
				x_1 = 0,\\
				x_2 + 1 = 0,\\
				x_3= 0,\\
				x_4 - x_7^2 + x_7 + 1 = 0,\\
				x_5 + 1= 0,
			\end{cases} 
			& &L^{(6)}_4 : \begin{cases}
				x_1 = 0,\\
				x_2 = 0,\\
				x_3 + 1 = 0,\\
				x_4 - x_5^2 + x_5 + 1 = 0,\\
				x_7 + 1  = 0,
			\end{cases}
		\end{align*}
		\begin{align*}
			L^{(7)}_1 : \begin{cases}
				x_1=0,\\
				x_2x_6 + x_6 + 1 = 0,\\
				x_3 + 1 = 0,\\
				x_4 + x_6= 0,\\
				x_5= 0,
			\end{cases} 
			\ L^{(7)}_2 : \begin{cases}
				x_1=0,\\
				x_2 + 1=0,\\
				x_3=0,\\
				x_4x_6 + x_6^2 + x_6 - 1 = 0,\\
				x_5 + x_6 = 0
			\end{cases}
			\ L^{(7)}_3 : \begin{cases}
				x_1 = 0,\\
				x_2 = 0,\\
				x_3 + 1=0,\\
				x_4 + x_5^2 + x_5 - 1=0,\\
				x_6 + 1 = 0.
			\end{cases}
		\end{align*}
		One can see that $\partial_{L^{(j)}_i}^{4,3}(\xi) = 0$ for all $i, j$. For example, let us compute the residue of $\xi$ at $L^{(2)}_1$. We have
        \begin{equation*}
            \begin{aligned}
                \partial_{L_1^{(2)}}^{4, 3}(\xi) &= \{x_1(L^{(2)}_1)\}_2 \otimes T_{L^{(2)}_1}\{x_2, x_3\} - \{x_2(L^{(2)}_1)\}_2 \otimes T_{L^{(2)}_1}\{x_1, x_3\} + \{x_3(L^{(2)}_1)\}_2 \otimes T_{L^{(2)}_1}\{x_1, x_2\}\\
                &= \{0\}_2 \otimes T_{L^{(2)}_1}\{x_2, x_3\} - \{x_2(L^{(2)}_1)\}_2 \otimes 1 + \{-1\}_2 \otimes T_{L^{(2)}_1}\{x_1, x_2\}\\
                &= 0
            \end{aligned}
        \end{equation*}
        since $\{0\}_2 = \{-1\}_2 = 0$ and 
        $$T_{L^{(2)}_1}\{x_1, x_3\} = (-1)^{\ord_{L^{(2)}_1}(x_1)\ \ord_{L^{(2)}_1}(x_3)} \ \dfrac{x_1^{\ord_{L^{(2)}_1}(x_3)}}{x_3^{\ord_{L^{(2)}_1}(x_1)}}(L^{(2)}_1) = (-1)^0\  \dfrac{x_1^0}{x_3^2}(L^{(2)}_1) = (-1)^2 = 1.$$
        Similarly, the residues of $\xi$ at the remaining curves $L^{(j)}_i$ are all trivial. By looking at equation \eqref{WPbar} of the compactification $\overline{W_P}$ in $\Pbb^1 \times \Pbb^1 \times \Pbb^1$, one can see that the poles of $x_1$ corresponds to the zeros of $x_1'$, which are obtained analogously and have the same form as the zeros of $x_1$. Hence the residues of $\xi$ at the poles of $x_1$ are also trivial. 
	\end{proof}

Now let us consider the Deninger chain. Recall that it is defined by
	\begin{equation*}
		\Gamma = V_P \cap  \{|x| = |y| = |z| = 1, |t|\ge 1\},
	\end{equation*}
    where $V_P = \{(x, y, z, t)\in (\Cbb^\times)^4 \ | \ P(x, y, z, t) = 0\}$. Notice that $V_P$ is nonsingular, hence the condition that $\Gamma \subset V_P^\reg$ is obviously satisfied. The boundary of $\Gamma$ is given by 
	\begin{equation*}
		\partial \Gamma = \{|x| = |y| = |z| =  |(x+1)(y+1)(z+1)| = 1\}.
	\end{equation*}
	Therefore, the boundary $\partial \Gamma$ does not contain any zero or pole of the functions $x, y, z, 1-x, 1-y, 1-z$.  And $\partial \Gamma$ does not contain any singular point of $W_P$. Then by Theorem \ref{7.2.8}, $\partial \Gamma$ defines an element in $H_2(X(\Cbb), \Qbb)^-$.  Using Maple, one can describe the boundary of the Deninger chain in polar coordinates $x = e^{it}, y = e^{is}, z = e^{ir}$ for $t, s, r \in [-\pi, \pi]$ as in Figure \ref{fig14.1}.
	\begin{figure}[h!]
		\centering
		\includegraphics[scale=0.4]{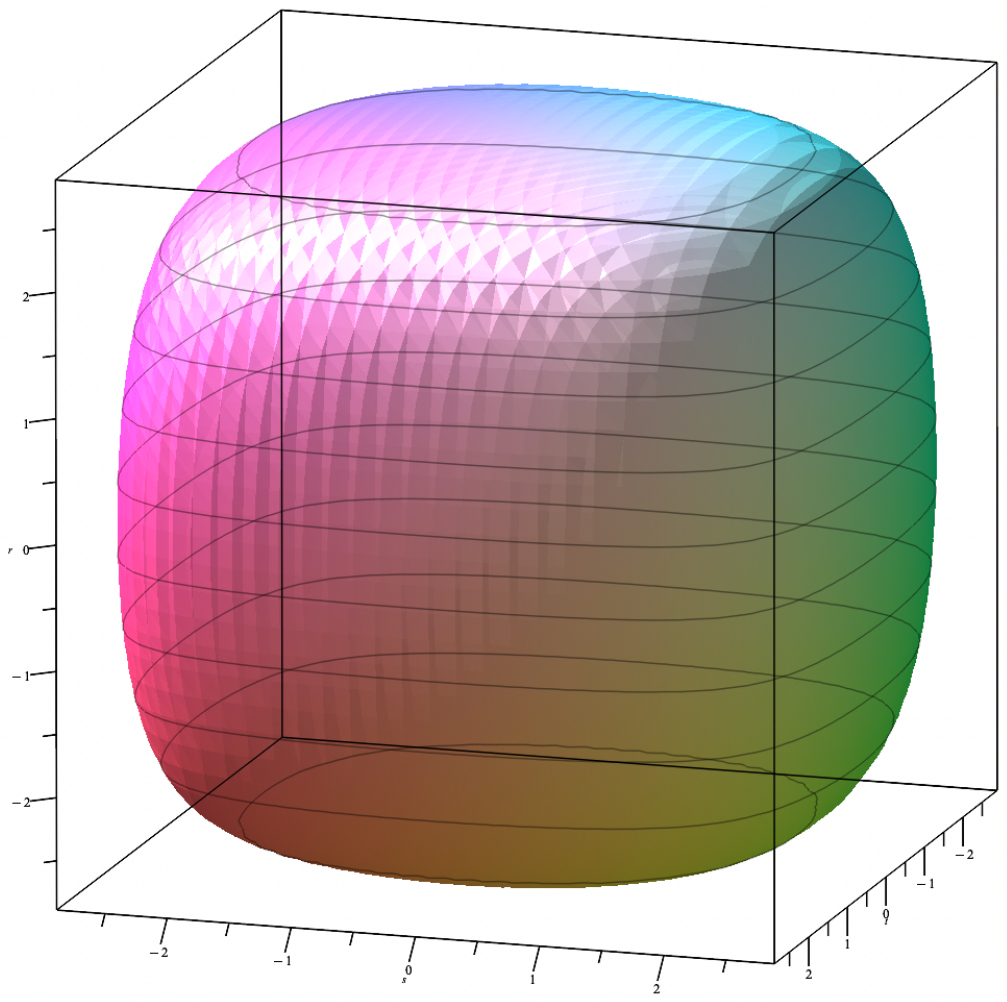}
		\caption{ The boundary of the Deninger chain.}
		\label{fig14.1}
	\end{figure}

Recall that $X$ is isomorphic to the elliptic $K3$ surface associated to $\Gamma_1(7)$ (see the proof of Proposition \ref{234}). In particular, $X$ has a canonical fibration $\pi: X \to \Pbb^1$ whose generic fiber is the following elliptic curve defined over $K = \Qbb(d)$
    \begin{equation*}
        E :  V^2 + (1+d-d^2)UV + (d^2 -d^3)V = U^3 + (d^2 -d^3)U^2.
    \end{equation*}
The singular fibers of $\pi$ are of Kodaira types $I_7, I_7, I_7, I_1, I_1, I_1$  and $E(K) \simeq \Zbb/7 \Zbb$. We then  have 
	\begin{equation*}
		|\det \Tbb(X)| =  \dfrac{\prod_{s \in \Sigma} m^{(1)}_s}{|E(K)_{tor}|^2} = \dfrac{7\cdot 7 \cdot 7 \cdot 1 \cdot 1 \cdot 1}{7^2} = 7,
	\end{equation*}
	where $m^{(1)}_s$ is the number of simple components in the singular fiber $X_s$ (see e.g., \cite[Lemma 1.3]{SI10}). By Theorem \ref{liv}, we have 
    $$L(h^2(X)_{\tr}, s) = L(f_7,, s),$$
     where $f_7$ is the unique modular form of weight 3 and level 7.
Hence, by Corollary \ref{final}, we obtain the following identity, under Goncharov's Conjecture \ref{15.1.3} and Beilinson's Conjecture  \ref{242} \begin{equation}\label{256}
		\m(P) = a \cdot L'(f_7, -1) + b \cdot \zeta'(-2), \quad a, b \in \Qbb.
	\end{equation}

\newpage
\fontsize{10pt}{10pt}\selectfont
\printbibliography

\end{document}